\documentclass[orivec,runningheads,envcountsame]{llncs}

\usepackage{graphicx}
\usepackage[leqno]{amsmath}
\usepackage{amssymb}

\usepackage{xcolor}

\usepackage{mathabx}

\newcommand{\polish} {{\mathbf{Pol}}}

\newcommand{\xcl} {{\cal L}}

 \newcommand{\xci} {{\cal I}}

\newcommand{\txllc}{\mbox{\scriptsize $\bot$}}
\newcommand{\txllf}[2]{\langle #1 , #2 \rangle}
\newcommand{\txllr}{\sqsubseteq}

\begin{document}             

\newcommand{\ccom}[1]{\mbox{\footnotesize{#1}}}

\setlength{\parindent}{0mm}
 \setlength{\parskip}{4pt}

\title{Notes on Interpretability between Weak First-order Theories: Theories of Sequences}

\author{Lars Kristiansen\inst{1,2}\\Juvenal   Murwanashyaka\inst{3}}

\authorrunning{Kristiansen and  Murwanashyaka}

\titlerunning{Notes on Interpretability \ldots}
 \institute{
Department of Mathematics, University of Oslo, Norway
\and
Department of Informatics, University of Oslo, Norway 
\and
Institute of Mathematics of the Czech Academy of Sciences, Prague, Czech Republic\\
		\email{larsk@math.uio.no} \email{murwanashyaka@math.cas.cz}
}

 \maketitle                   

\newcommand{\xd}{\texttt{D}}
\newcommand{\integer}{\ensuremath{\mathbb Z}}
\newcommand{\rational}{\ensuremath{\mathbb Q}}
\newcommand{\nat}{\ensuremath{\mathbb N}}
\newcommand{\real}{\ensuremath{\mathbb R}}
\newcommand{\kleeneT}{{\mathcal T}}

\newcommand{\kleeneU}{{\mathcal U}}

\newcommand{\llc}{e}
\newcommand{\tzero}{\dot{\mathrm{0}}}

\newcommand{\llf}[2]{#1  \! \vdash \!\!   #2 }
\newcommand{\llapp}[2]{#1 \circ #2 }

\newcommand{\tsuc}[1]{\dot{\mathrm{S}} #1 }
\newcommand{\llr}{\sqsubseteq}

\newcommand{\fintnavn}{\mathsf{Seq}} 
\newcommand{\wfintnavn}{\mathsf{WSeq}} 
\newcommand{\folast}{\mathsf{AST}}
 \newcommand{\folastext}{\mathsf{AST + EXT}}
\makeatletter
\newcommand{\leqnomode}{\tagsleft@true}
\newcommand{\reqnomode}{\tagsleft@false}
\makeatother

\begin{abstract}
We introduce a first-order theory $\fintnavn$ which is mutually interpretable with Robinson's $\mathsf{Q}$.
The universe of a standard model for $\fintnavn$ consists of sequences.
We prove that $\fintnavn$ directly interprets the adjuctive set theory $\mathsf{AST}$, and we prove that $\fintnavn$  interprets 
the tree theory $\mathsf{T}$ and the set theory $\mathsf{AST + EXT}$. 
\end{abstract}

\section{Introduction}

First order theories like Robinson's $\mathsf{Q}$, Grzegorczyk's $\mathsf{TC}$ and adjunctive set theory 
$\mathsf{AST}$ serve as important metamathematical tools. These are natural theories given by a handful of transparent axioms. They are all mutually interpretable with each other and also with a number of other natural theories, e.g., the 
tree theories studied  in Kristiansen \& Murwanashyaka \cite{treteori}, 
Damnjanovic \cite{dam} \cite{damto} and Murwanashyaka \cite{jtrees}, and 
the concatenation theories studied in  Murwanashyaka \cite{juvenalaml}.

At a first glance all these theories might seem very weak, and they are indeed weak, but it turns out that  they provide
 the building blocks needed to construct (interpret, encode, formalise, emulate, reconstruct, \ldots pick your choice) substantial mathematics.\footnote{
 See Ferreira \&  Ferreira \cite{ff} for more on how mathematics can be constructed in these theories.}
 This (maybe somewhat unexpected) strength 
stems  from the theories'   ability to represent  sequences. Intuitively, access to sequences seems to be
 both necessary and
sufficient in order to construct substantial parts of mathematics. Now, it is rather tricky to deal with sequences in weak number theory ($\mathsf{Q}$), see Nelson \cite{nelson} and Chapter V of Hajek \& Pudlak \cite{hp};
neither is it all that straightforward in concatenation theory ($\mathsf{TC}$) as we encounter the growing comma 
problem, see Quine \cite{quine}  and Visser \cite{growing}. It might be easier in weak set theory ($\mathsf{AST}$), but still it is far from trivial.

In this paper we introduce $\fintnavn$ which is designed to be a minimal first-order theory mutually interpretable with 
$\mathsf{Q}$ (it is minimal in the sense that it will not interpret 
$\mathsf{Q}$ if we remove any of its five axioms). Moreover, it is designed to provide, as directly as possible, the means needed to represent sequences.
This implies that $\fintnavn$ in some respects may be a more natural starting point for a construction of mathematics than the theories discussed above. $\fintnavn$ provides directly the basic building blocks we need to represent sequences.
In the other theories we need to put effort into building these basic building blocks. Of course, when we work in 
$\fintnavn$, we have to represent, let us say, natural numbers by sequences, but it is both easier and more natural to
represent natural numbers by sequences than it is to represent sequences by natural numbers.

We do not claim that $\fintnavn$ is more natural than other theories in any absolute sense. A theory will of course
have its own 
distinctive features, and e.g., $\mathsf{Q}$ is  a very natural starting point for Nelson's investigations into predicative
arithmetic \cite{nelson}, and $\mathsf{TC}$ is a very natural theory from Grzegorczyk's \cite{sl} \cite{gz} point of view since he wants to ``formulate
the proof of undecidability not on the grounds of arithmetic but directly on a theory of text'' \cite{sl}. 
We just claim that $\fintnavn$ is a natural theory from a certain point of view, and it is not on our agenda 
to discuss if this naturalness is a matter of taste or a matter of deeper (philosophical or mathematical) nature.

\paragraph{An overview of the paper.}
 In Section \ref{savnermor} we introduce the theory $\fintnavn$, together with a weaker variant
called $\wfintnavn$, and prove that the two theories have a certain strength.
In Section \ref{savnernoe} we show that $\fintnavn$ directly interprets $\mathsf{AST}$, and thus, $\fintnavn$
is a sequential theory (a theory is {\em sequential} by definition if it directly interprets $\mathsf{AST}$).
In Section \ref{gulspisser} we develop a basic technical machinery which  will be needed in Section \ref{gronnspisser}
where we prove that $\fintnavn$ interprets the tree theory $\mathsf{T}$. 
Finally, in Section \ref{blankspisser}, we build on the work in Section \ref{gronnspisser} and prove that $\fintnavn$ interprets
 $\mathsf{AST + EXT}$.

The techniques used in Section \ref{gronnspisser} and Section \ref{blankspisser} might be of some independent interest:
We represent finite binary trees and hereditarily finite sets by certain sequence of natural numbers which we will
call {\em snakes}. 

It should be pretty clear that the theorem below can be proved
by using  the techniques described in   Chapter V of Hajek \& Pudlak \cite{hp}. Thus, we state the theorem with no further
comment and  offer no proof.

\begin{theorem} \label{brillestell} 
$\mathsf{Q}$ interprets $\fintnavn$.
\end{theorem}

\paragraph{Discussion of some related work.} Our work should be compared to some work of Damnjanovic.
Prior to this work of Damnjanovic, it was known that $\mathsf{AST}$ interprets $\mathsf{Q}$.

In \cite{slatanonast}, Damnjanovic 
introduces the concatenation theory
$\mathsf{QT}$ and proves that $\mathsf{Q}$ interprets $\mathsf{QT}$ and that $\mathsf{QT}$ interprets 
$\mathsf{AST + EXT}$. 
Thus, since $\mathsf{AST}$ interprets $\mathsf{Q}$, it follows from the results in \cite{slatanonast} that 
$\mathsf{AST}$  interprets  $\mathsf{AST+ EXT}$. By the same token, just replace $\mathsf{QT}$ with $\fintnavn$,
 it follows from our results that
$\mathsf{AST}$  interprets  $\mathsf{AST+ EXT}$.

Kristiansen \& Murvanashyaka \cite{treteori} introduces the tree theory $\mathsf{T}$. They prove that $\mathsf{T}$
interprets $\mathsf{Q}$ and conjecture that  $\mathsf{Q}$ interprets $\mathsf{T}$.  Damnjanovic \cite{dam}
proves that $\mathsf{Q}$ indeed interprets $\mathsf{T}$ by proving that $\mathsf{QT}$  interprets $\mathsf{T}$.
We prove that $\fintnavn$  interprets $\mathsf{T}$, and thus it also follows the results in this paper that 
$\mathsf{Q}$ interprets $\mathsf{T}$. 

These considerations show that there are certain similarities between our work and the work of Damnjanovic:
when Damnjanovic resorts to $\mathsf{QT}$, we resort to $\fintnavn$. To the authors, these similarities have become
apparent in hindsight, and they may reflect that $\mathsf{QT}$ and $\fintnavn$ share some salient  features.

\paragraph{Some considerations on text, alphabet strings and sequences.} 
Concatenation theory is a theory about strings over an alphabet, and occasionally  
concatenation theory is considered as a mathematical model of text, e.g., in  
Grzegorczyk's \cite{sl}. We can also consider the elementary theory of sequences which we introduce in this paper
as a mathematical model of text, and perhaps even a more adequate model than concatenation theory. 

A text 
will be written in a language, and every language has 
an inherent hierarchical structure. Such a structure can be fairly well modelled 
by a context-free grammar. This is true of formal languages used to describe, e.g.,  a deduction in a proof calculus
or the normalisation of a $\lambda$-term; this is true of semi-formal languages
like the ones used to express formulas and equations in mathematical textbooks;
and this is also more or less true of natural languages like Norwegian and Kinyarwanda.
The universe of the standard structure for $\fintnavn$ consists of sequences with a hierarchical structure, indeed
the sequences in the universe can be generated by the archetypical context-free grammar
$S \rightarrow ()$;  $S \rightarrow (S)$; $S \rightarrow SS$.
Such sequences may be more suited to modelling the structure of a language, and thus the structure of text,
than the flat strings of concatenation theory.

\section{The theories $\fintnavn$ and $\wfintnavn$}

\label{savnermor}

Let $\xcl$ be the first order language $\{e , \ \llf{ }{ } \ , \circ \}$ where  $\llf{ }{ }$ and $\circ$ are binary
function symbols and $e$ is a constant symbol. The   $\xcl$-theory $\fintnavn$  is given by the five  axioms in Figure \ref{fintnavnaksiomer}.

\begin{figure}[t] 
$$\begin{array}{lll}
 \fintnavn_1 & \; & \forall x y [  \  \llf{x}{y} \neq \llc   \   ]  \\
 \fintnavn_2 & \; &  \forall x_1 x_2 y_1 y_2 [  \  \llf{x_1}{x_2} = \llf{y_1}{y_2} \; \rightarrow \; (  \  x_1 =y_1  \, \wedge \, x_2 = y_2 \  )       \  ]  \\
 \fintnavn_3   & \; & \forall x  [  \  \llapp{x}{\llc} =  x \   ] \\
 \fintnavn_4   & \; & \forall x y z   [  \  \llapp{x}{(\llf{y}{z})}   = \llf{(\llapp{z}{y})}{z} \     ] \\
 \fintnavn_5  & \; & \forall x    [  \  x= \llc \; \vee \; \exists y z [ \    x= \llf{y}{z}      \ ] \     ]
\end{array}
$$
\caption{The non-logical axioms of $\fintnavn$.
\label{fintnavnaksiomer}}
\end{figure}

Next we define the set of sequences inductively: the empty sequence $()$ is a sequence.
For any $n>0$,  if $s_1$, $s_2$, \ldots and $s_n$ are sequences, then $(s_1, s_2, \ldots, s_n)$
is a sequence.

We define the standard model $\mathfrak{S}$ for the theory $\fintnavn$. The universe of 
$\mathfrak{S}$ is the set of all sequences (as defined above). Furhermore, $e^\mathfrak{S}= ()$;
the operator $\vdash^\mathfrak{S}$ appends an element to a sequence, that is
$$
(s_1, s_2, \ldots, s_n)\vdash^\mathfrak{S} (t_1, t_2, \ldots, t_m) \; = \; (s_1, s_2, \ldots, s_n,(t_1, t_2, \ldots, t_m))
$$
for any sequence $(s_1, s_2, \ldots, s_n)$ and any sequence $(t_1, t_2, \ldots, t_m)$; the operator $\circ^\mathfrak{S}$ 
concatenates two sequences, that is
$$
(s_1, s_2, \ldots, s_n)\circ^\mathfrak{S} (t_1, t_2, \ldots, t_m) \; = \; 
(s_1, s_2, \ldots, s_n,t_1, t_2, \ldots, t_m)
$$
for any sequence $(s_1, s_2, \ldots, s_n)$ and any sequence $(t_1, t_2, \ldots, t_m)$.

For any sequence $s$, we define the {\em sequeral} $\overline{s}$ by
$$
\overline{()}=e \;\;\; \mbox{ and } \;\;\; \overline{(s_1,  \ldots, s_n)} =
\llf{(\ldots ( \llf{(\llf{e}{\overline{s_1}})}{\overline{s_2}})\ldots )}{\overline{s_n}} \; .
$$
The sequerals serve as canonical names for the elements in the universe of the  standard model.

For any $\xcl$-terms $t_1,t_2$ the formula $t_1 \llr t_2$ is shorthand for $\exists y [t_1 \circ y = t_2]$,
and the formula $t_1 \not\llr t_2$ is shorthand for  $\neg t_1 \llr t_2$.
  The formula $\forall x\llr t [\phi]$ is shorthand for $\forall x[x\llr t \rightarrow \phi]$.
We define the $\Sigma$-formulas inductively by
\begin{itemize}
\item  $\phi$ and $\neg \phi$ are  $\Sigma$-formulas if $\phi$ is an atomic $\xcl$-formula
\item $ s  \llr t  $ and $ s \not\llr t $ are    $\Sigma$-formulas if $s$ and $t$ are $\xcl$-terms
\item $\alpha \wedge \beta$ and $\alpha \vee \beta$ are $\Sigma$-formulas
 if $\alpha$ and $\beta$ are $\Sigma$-formulas
\item $\exists x [\phi]$ is a $\Sigma$-formula if $\phi$ is a $\Sigma$-formula
\item $\forall x\llr t [\phi]$ is a $\Sigma$-formula if $\phi$ is a $\Sigma$-formula, $x$ is a variable 
 and $t$ is $\xcl$-term that not contains $x$.
\end{itemize}

The first-order theory $\wfintnavn$ is given by the axioms in Figure \ref{wfintnavnaksiomer}.

\begin{figure}[t] 
The non-logical axioms of  $\wfintnavn$  are given by the three  axioms schemes:
\begin{align}
 \overline{s}\neq \overline{t} \tag{$\wfintnavn_1$}
\end{align}
where $s$ and $t$ er distinct sequences.
\begin{align}
\overline{(s_1,  \ldots, s_n)}\circ \overline{(t_1,  \ldots, t_m)} \; = \; 
\overline{(s_1,  \ldots, s_n,t_1,  \ldots, t_m)}
  \tag{$\wfintnavn_2$}
\end{align}
for any sequences $(s_1,  \ldots, s_n)$ and  $(t_1,  \ldots, t_m)$.
\begin{align}
 \forall x  [  \  x \sqsubseteq  \overline{s} \;\; \rightarrow \;\; \bigvee_{ t\in \xci(s) } x = \overline{t}  \ ]\tag{$\wfintnavn_3$}
\end{align}
where  $s$ is a sequence and $\xci(s)$ is the set of all initial segments of $s$.
\caption{The non-logical axioms of $\wfintnavn$.
\label{wfintnavnaksiomer}}
\end{figure}

\begin{lemma} \label{abracadabra} 
Let $t$ be  any variable-free $\xcl$-term $t$, and let $s$ be the sequence that is the interpretation of $t$ 
in the structure $\mathfrak{S}$. Then we have  
$\wfintnavn_2 \vdash t = \overline{s}$.
\end{lemma}

\begin{proof}
This lemma is proved by induction on the structure of the term $t$. Thus, the induction cases are $t\equiv e$ and 
$t\equiv  \llf{t_1}{t_2}$ and $t\equiv t_1 \circ t_2$. Everything should work straightforwardly.
\qed
\end{proof}

\begin{theorem}[$\Sigma$-completeness of $\wfintnavn$] 
For any $\Sigma$-sentence $\phi$, we have
$$
\mathfrak{S}\models \phi \;\;\;  \Rightarrow  \;\;\; \wfintnavn \vdash \phi \; .
$$
\end{theorem}

\begin{proof}
Assume $\mathfrak{S}\models \phi$.
We prove $\wfintnavn \vdash \phi$ by induction on the structure of a $\Sigma$-sentence $\phi$.

Let $\phi$ be atomic, that is, let $\phi$ be of the form $t_1 = t_2$. Then our theorem follows by  Lemma \ref{abracadabra}.

When $\phi$ is the negation of an atomic sentence, the theorem follows by Lemma \ref{abracadabra} and $\wfintnavn_1$.

The theorem follows easily by using the induction hypothesis when $\phi$ is of the form $\alpha \wedge \beta$ or of the form
 $\alpha \vee \beta$.

 Let $  \phi $ be of the form $ s  \llr  t $. 
Then, $ \mathfrak{S}\models  t = s r $ for some  variable-free term  $r$. 
By  Lemma \ref{abracadabra}  and $\wfintnavn_2$, we get $ \wfintnavn \vdash t = s r $,   which implies  $ \wfintnavn \vdash  s    \llr  t $.

  Let $  \phi $ be of the form $ s  \not \llr  t $.
  Let $r $ be the unique sequence such that   $\mathfrak{S}\models t = \overline{r}$. 
  Then,  $ \mathfrak{S}\models  \bigwedge_{ w \in \xci(r) }  s \neq  \overline{w}      $  (recall that   $\xci(r)$ is the set of all initial segments of $r$). 
Hence,  $  \wfintnavn \vdash  \bigwedge_{ w \in \xci(r) }     s \neq   \overline{w}     $. 
Thus,   $  \wfintnavn \vdash   s  \not \llr  t $ by Lemma   \ref{abracadabra} and $\wfintnavn_3$.

 Assume $\phi$ is of the form $\exists x [\psi(x)]$. Then there exists a sequence $s$ 
 such that $\mathfrak{S}\models \psi(\overline{s})$. By the induction hypothesis, we have $\wfintnavn\vdash \psi(\overline{s})$,
 and thus we also have $\wfintnavn\vdash \exists x [\psi(x)]$.
 
 Finally, assume that $\phi$ is of the form $\forall x \sqsubseteq  t [\psi(x)]$.  Let $s$ be the  unique 
 sequence such that $\mathfrak{S}\models t=\overline{s}$ and let $s_1,\ldots,s_k$ be the initial segments of $s$.
 Then we have $\mathfrak{S}\models \psi(\overline{s}_i)$ for $i=1,\ldots, k$. By the induction hypothesis, we have
  $\wfintnavn\vdash \psi(\overline{s_i})$, and thus also $\wfintnavn\vdash x=\overline{s_i} \rightarrow \psi(x)$, for $i=1,\ldots, k$.
  Furthemore, by $\wfintnavn_3$, we have $\wfintnavn\vdash x \sqsubseteq  \overline{s} \rightarrow \psi(x)$, and by
Lemma \ref{abracadabra}, we get $\wfintnavn\vdash x \sqsubseteq  t \rightarrow \psi(x)$. Thus, we have 
$\wfintnavn\vdash \forall x\llr t [\psi(x)]$ since $\forall x \sqsubseteq  t [\psi(x)]$ simply is  an abbreviation for 
$\forall x[ x \sqsubseteq  t \rightarrow \psi(x)]$.
\qed
\end{proof}

We conjecture  that $\wfintnavn$ is mutually interpretable with Robinson's $\mathsf{R}$.

\begin{lemma} \label{kongharald}
 For any sequence $s$, we have 
$$\fintnavn  \vdash  \forall x  \;  [   \    x  \sqsubseteq    \overline{s} \;\; \leftrightarrow  \;\;   \bigvee_{ t \in  \xci(s)  }   x = \overline{t}   \    ]  $$
where $ \xci(s)  $ is the set of all  initial segments of $s$. 
\end{lemma}

\begin{proof}
We prove  the lemma by induction on the complexity of the sequence $s$. 

Assume $s = ()$. Then, $  \overline{s} = e $ and $ \bigvee_{ t \in  \xci(s)  }   x = \overline{t} $ will simply be the formula $x=e$.
Thus, we  have to prove that $\fintnavn  \vdash x  \sqsubseteq e \leftrightarrow x=e$.
In order to prove the right-left implication, assume that $x=e$. Then we need to prove that $e \sqsubseteq e$.
This holds since we have $ e \circ e = e $ by $  \fintnavn_3 $. In order to prove the left-right implication
assume that $ x  \sqsubseteq e $. Then   there exists $ w $ such that $ e = x \circ w $. 
By $ \fintnavn_5 $, we can split the proof into the case  (i) $ w = e $ 
and the case  (ii) $ w = \llf{v}{u} $ for some $ v, u $. 
In case  (i), we have $ x = e $ by $ \fintnavn_3$. 
In case (ii) we have a contradiction:  By $ \fintnavn_4 $ we have   
$ e = x \circ (\llf{v}{u}   ) = \llf{(x \circ v)} {u} $ 
  which contradicts $ \fintnavn_1 $. 
Thus, $ x \sqsubseteq  e $ if and only if $ x = e $.

Assume $ s = ( s_1,   \ldots , s_{n+1}  ) $.  Let  $ s^{ \prime } = (  s_1,  \ldots , s_n  ) $. 
Observe that $  \xci (s) =  \xci (s^{  \prime } )  \cup \lbrace s  \rbrace $  and 
$  \overline{s} =  \llf{\overline{ s^{  \prime } } } {   \overline{ s_{ n+1} } } $. 
Clearly, $ t \in  \xci (s) $ implies  $ \overline{t} \sqsubseteq  \overline{s} $, and thus $\bigvee_{ t \in  \xci(s)  }   x = \overline{t}$
implies $ x  \sqsubseteq    \overline{s}$.
In order to prove the left-right implication
 assume that $ x  \sqsubseteq \overline{s} $.  Then  there exists $ w $ such that $  \overline{s} = x \circ w $. 
By $ \fintnavn_5 $, we split the proof into the case (i) $ w = e $ and the case (ii) 
$w= \llf{v}{u}   $ for some $ v, u $. 
In case  (i), we have $ x =  \overline{s} $ by $ \fintnavn_3$. 
In case  (ii),  we have   $ \overline{ s^{  \prime } } = x \circ v $ by  $ \fintnavn_4 $ and  $ \fintnavn_2 $, 
which by the induction hypothesis,   implies $ x =  \overline{t} $ for some $ t \in \xci (s) $.
Thus,  $ x \sqsubseteq  \overline{s} $ if and only if   $ x =  \overline{t} $ for some $ t \in \xci (s) $. 
\qed
\end{proof}

\begin{theorem} \label{kongolav}
The theory $\fintnavn$ is an extension of the theory $\wfintnavn$, that is, we have
$$
\wfintnavn \vdash \phi \;\;\;  \Rightarrow  \;\;\; \fintnavn \vdash \phi \; .
$$
for any $\xcl$-formula $\phi$.
\end{theorem}

\begin{proof}
It is easy to see that $\fintnavn$ proves any instance  of the axiom scheme $\wfintnavn_1$ (use  $\fintnavn_1$ and  
$\fintnavn_2$).
We prove that $\fintnavn$ proves any  instance of  $\wfintnavn_2$ by induction on the length of 
the sequence $(t_1, \ldots , t_m)$. Let $(t_1, \ldots , t_m)$ be the empty sequence $()$. Then
$\overline{(t_1, \ldots , t_m)}=e$, and $\wfintnavn_2$ holds by $\fintnavn_3$. Let $m>0$. Then
by $\fintnavn_4$ and the induction hypothesis, we have
\begin{multline*}
\overline{(s_1,  \ldots, s_n)}\circ \overline{(t_1,  \ldots, t_m)}  \; = \; 
\overline{(s_1,  \ldots, s_n)}\circ ( \ \llf{\overline{(t_1,  \ldots, t_{m-1})}}{\overline{t_m}} \ )
      \\ \; = \; 
 \llf{  (    \  \overline{(s_1,  \ldots, s_n)}\circ  \overline{(t_1,  \ldots, t_{m-1})}  \   )  }{\overline{t_m}}
\\ \; = \;   \llf{  \overline{(s_1,  \ldots, s_n, t_1,  \ldots, t_{m-1})}} {\overline{t_m}}      
\; = \;    \overline{(s_1,  \ldots, s_n,t_1,  \ldots, t_m)}\; .
\end{multline*}
Finally,  $\fintnavn$ proves any instance  of  $\wfintnavn_3$  by Lemma \ref{kongharald}.
\qed
  \end{proof}

\begin{corollary}
 The theory $\fintnavn$ is $\Sigma$-complete.
\end{corollary}

From now on, we may skip the concatenation operator  in first-order formulas and simply write $st$ in place of $s\circ t$.

\section{$\fintnavn$ is a sequential theory}

\label{savnernoe}

The  theory $\folast$  is given by the two  axioms $\folast_1$ and $\folast_2$ in Figure \ref{astaksiomer}.
We prove that $\fintnavn$  directly interprets $\folast$.

\begin{figure}[t] 
$$\begin{array}{llll}
 \folast_1 & \; & \exists y \forall x [  \  x \not\in y   \   ] & \mbox{(empty set)} \\
 \folast_2 & \; &  \forall x y \exists z \forall u [  \  u\in z \; \leftrightarrow \; (  \  u\in x  \, \vee \, u=y \  )   \  ]  \;\;\;
 & \mbox{(adjunction)} \\
 \folast_3 & \; &  \forall x y  [ \  \forall z [ \  z\in x \, \leftrightarrow \, z\in y \ ]   \; \rightarrow \;  x=y  \  ]  \; .
 & \mbox{(extensionality)} \\
\end{array}
$$
\caption{The non-logical axioms of $\folastext$. The two first axioms are the axioms of   $\folast$.
\label{astaksiomer}}
\end{figure}

We define the direct interpretation 
$\cdot^\tau$.
Let $(x\in y)^\tau = \exists v_1 v_2[(\llf{v_1}{x})v_2=y]$. 
Then, we have
$$
(\folast_1)^\tau \;\; = \;\; \exists y \forall x \neg  \exists v_1 v_2[  \ (\llf{v_1}{x})v_2=y   \   ] 
$$
and
\begin{multline*}
(\folast_2)^\tau \;\; = \;\;\forall x y \exists z \forall u [  \   \exists v_1 v_2[(\llf{v_1}{u})v_2=z]
 \; \leftrightarrow \;  \\ (  \  \exists v_1 v_2[(\llf{v_1}{u})v_2=x]  \, \vee \, u=y \  )       \  ]  \; .
\end{multline*}

\begin{lemma} \label{ladygaga} 
$\fintnavn \vdash (\folast_1)^\tau$.
\end{lemma}

\begin{proof}
We will prove that 
\begin{align} \label{lillefinger}
\fintnavn \vdash \forall x v_1 v_2 [(\llf{v_1}{x})v_2\neq \llc ] \; .
\end{align}
The lemma follows from (\ref{lillefinger}) by pure first-order logic.

Let $v_1, x , v_2$ be arbitrary. By $\fintnavn_5$, we have $v_2= e$ or $v_2 = \llf{w_1}{w_2}$ for some elements $w_1, w_2$.

Assume that $v_2 =e$. By $\fintnavn_3$, we have $(\llf{v_1}{x})v_2 = \llf{v_1}{x}$, and thus, by $\fintnavn_1$, we have 
$(\llf{v_1}{x})v_2 \neq e$.

Assume that $v_2 = \llf{w_1}{w_2}$. By $\fintnavn_4$, we have
$$(\llf{v_1}{x})v_2 = (\llf{v_1}{x} )(\llf{w_1}{w_2}) =  \llf{((\llf{v_1}{x})w_1)} {w_2} $$
and thus, by $\fintnavn_1$, we have 
$(\llf{v_1}{x})v_2 \neq e$.
\qed
\end{proof}

\begin{lemma} \label{eltonjohn} 
$\fintnavn \vdash (\folast_2)^\tau$.
\end{lemma}

\begin{proof}
We will prove that
\begin{align} \label{ringfinger}
\fintnavn \vdash  \exists v_1 v_2[ \ (\llf{v_1}{u})v_2=\llf{x}{y} \ ]
 \; \rightarrow \;   (  \  \exists v_1 v_2[ \ (\llf{v_1}{u})v_2=x \ ]  \, \vee \, u=y \  ) 
\end{align}
and
\begin{align} \label{langfinger}
\fintnavn \vdash  (  \  \exists v_1 v_2[ \ (\llf{v_1}{u})v_2=x \ ]   \, \vee \, u=y \  )  
 \; \rightarrow \; \exists v_1 v_2[ \ (\llf{v_1}{u})v_2=\llf{x}{y} \ ] \; . 
 \end{align}
 The lemma follows from (\ref{ringfinger}) and (\ref{langfinger}) by pure  logic.
 
 First we prove (\ref{ringfinger}). We assume that $v_1, v_2$ are such that $(\llf{v_1}{u})v_2=\llf{x}{y}$.
 By $\fintnavn_5$, we have $v_2 = e$ or $v_2 = \llf{w_1}{w_2}$ for some $w_1,w_2$.
 
 Assume $v_2=e$. By $\fintnavn_3$, we have $\llf{x}{y}=(\llf{v_1}{u})v_2=\llf{v_1}{u}$, and thus we have $u=y$ by $\fintnavn_2$.
 
 Assume $v_2 = \llf{w_1}{w_2}$. By $\fintnavn_4$ (third equality), we have
 $$
 \llf{x}{y} = (\llf{v_1}{u})v_2 = (\llf{v_1}{u})(\llf{w_1}{w_2}) = \llf{((\llf{v_1}{u})w_1)}{w_2}\; .
 $$
 By $\fintnavn_2$, we have $x=(\llf{v_1}{u})w_1$, and hence, we have $\exists v_1,v_2 [  (\llf{v_1}{u})v_2=x ]$. This proves
 (\ref{ringfinger}).

 We turn to the proof of (\ref{langfinger}). 
 Assume we have $(\llf{w_1}{u})w_2=x$ fore some elements $w_1,w_2$.
 By $\fintnavn_4$, we have $$\llf{x}{y} = \llf{(\llf{w_1}{u})w_2}{y} = (\llf{w_1}{u})(\llf{w_2}{y})$$
 and hence, we have $\exists v_1 v_2[ (\llf{v_1}{u})v_2=\llf{x}{y}]$.
 
 Assume that $u=y$. By $\fintnavn_3$, we have $(\llf{x}{y})e = \llf{x}{y}$, and hence, we have 
 $\exists v_1 v_2[ (\llf{v_1}{u})v_2=\llf{x}{y}]$.
 This proves
 (\ref{langfinger}).
 \qed
\end{proof}

\begin{theorem}
 (i) $\fintnavn$ is a sequential theory. (ii) $\fintnavn$ is essentially undecidable.
\end{theorem}

\begin{proof}
It follows from Lemma \ref{ladygaga} and Lemma \ref{eltonjohn} that the theory directly  interprets $\folast$.
Thus, the theory is seqential, moreover, the theory is essentially undecidable since $\folast$ is essentially undecidable.
\qed
\end{proof}

\section{Indexed Sequences}

\label{gulspisser}

\subsection{Terminology and Notation}

A {\em class} is a formula with a free variable. 
Let $ \phi $ be a formula and  let $ K$ be a class. We define   $  \phi^K $ inductively by
 $  \phi^K = \phi $ if $  \phi $ is an atomic formula; 
$ ( \neg  \phi )^K = \neg   \phi^K $;
$ (  \phi  \oslash  \psi )^K =  \phi^K  \oslash   \psi^K $   for $  \oslash  \in  \lbrace  \vee,  \wedge ,  \rightarrow , \leftrightarrow  \rbrace  $;
$$  (   \forall x  \;   \phi  )^K =  \forall   x \;   [   \    K(x) \rightarrow    \phi^K    \   ]     
 \;  \;  \;  \mbox{ and }  \;  \;  \; 
  (   \exists x  \;   \phi  )^K =  \exists   x \;   [   \    K(x)     \;   \wedge   \;     \phi^K    \   ]   \; . $$ 
The formula $  \phi^K $ is called {\em the restriction of $  \phi $ to the class $K$}.

When $K$ is a class,
we may write $x \in K $ in place of  $ K(x) $ and use  standard set-theoretic notation. We
 may also  write $  \forall x \in K    [    \phi     ]  $  and $  \exists  x \in K     [    \phi    ]  $  in  place of
 $   \forall   x     [    x \in K  \rightarrow    \phi      ]    $  
and 
$  \exists   x   [   x \in K        \wedge       \phi      ]    $, 
respectively.

\subsection{A Variant of $ \mathsf{Seq} $  }

Let  $  \mathsf{Seq}^* $ denote the theory we obtain by taking $  \mathsf{Seq} $ and replacing $  \mathsf{Seq}_3 $ and  $  \mathsf{Seq}_5 $  with the axioms
\[
\begin{array}{r c l }
\mathsf{Seq}^*_3  &    &   
\forall x      [     \     xe = x  \;  \wedge  \;   ex = x    \      ]  
\\
\mathsf{Seq}^*_5  &    &   
\forall x y z w       \big[ \  xy = zw   \;\;  \leftrightarrow   \;\;
  \exists u [ \, ( z=xu   \,  \wedge   \,      uw = y )   \,   \vee   \,    (    x = zu     \,  \wedge   \,      uy = w     )   \,   ]     \         \big]\, . 
\end{array}
\]
The right-left implication of $ \mathsf{Seq}^*_5 $  is logically equivalent to associativity of $  \circ $. 
In the context of concatenation theories, the  left-right implication  is called the \emph{editor axiom}  and is attributed to Alfred Tarski.

\begin{figure}[t] 
$$\begin{array}{lll}
 \fintnavn_1 & \; & \forall x y [  \  \llf{x}{y} \neq \llc   \   ]  \\
 \fintnavn_2 & \; &  \forall x_1 x_2 y_1 y_2 [  \  \llf{x_1}{x_2} = \llf{y_1}{y_2} \; \rightarrow \; (  \  x_1 =y_1  \, \wedge \, x_2 = y_2 \  )       \  ]  \\
 \fintnavn^*_3   & \; &   \forall x      [     \     xe = x  \,  \wedge  \,   ex = x    \      ]      \\
 \fintnavn_4   & \; & \forall x y z   [  \  \llapp{x}{(\llf{y}{z})}   = \llf{(\llapp{z}{y})}{z} \     ] \\
 \fintnavn^*_5  & \; & \forall x y z w       \big[ \  xy = zw   \;\;  \leftrightarrow   \;\;
  \exists u [ \, ( z=xu   \,  \wedge   \,      uw = y )   \,   \vee   \,    (    x = zu     \,  \wedge   \,      uy = w     )   \,   ]     \         \big] 
\end{array}
$$
\caption{The non-logical axioms of $\fintnavn^*$.
\label{fintnavnstjerneaksiomer}}
\end{figure}

In this section, we show that  $  \mathsf{Seq} $ interprets $  \mathsf{Seq}^* $. 
In the next  section,   we show  that 
$ \mathsf{Seq} $ interprets the theory   we obtain by extending $  \mathsf{Seq} $ with $  \mathsf{Seq}^*_3  $ and $  \mathsf{Seq}^*_5 $. 
It  is not clear to us whether $  \mathsf{Seq}^* $ is a sequential theory, that is, 
whether $  \mathsf{Seq}^*$ directly interprets $  \mathsf{AST} $. 
Visser \cite{growing} has shown that Grzegorczyk's theory of concatenation $  \mathsf{TC} $, which has $  \mathsf{Seq}^*_5 $  as a non-logical axiom, 
is not sequential by showing that it does not even have pairing. 
In contrast,  $  \mathsf{Seq}^* $  has  pairing since $  \vdash $ is a pairing function  by $  \mathsf{Seq}_2 $.

\begin{problem}

Is $  \mathsf{Seq}^*$ sequential? 

\end{problem}

\begin{lemma}

There exists a class $  \mathcal{J} $ such that $ \mathsf{Seq} \vdash \phi^{  \mathcal{J} } $ for each axiom $  \phi $ of $  \mathsf{Seq}^* $. 
Furthermore, let  $  x \preceq  y  $ be shorthand for  $  \exists z \in  \mathcal{J}   [       y = xz       ]  $. 
Then, $  \preceq $ is reflexive and transitive,    $  \forall  y   \in  \mathcal{J}     \forall x  \preceq   y      [        x  \in  \mathcal{J}        ]    $
 and $  \forall w   [   \llf{e}{w}  \in     \mathcal{J}          ]    $.  

\end{lemma}

\begin{proof}

It suffices to define a domain  $J$ such that   the axioms of  $  \mathsf{Seq}^* $ hold restricted to $J$, 
downward closure under $  \preceq $ will be a consequence of how $J$ is defined. 
We start by defining an  intermediate class  $ A    \supseteq   J $.

Let  $A$ denote the class of  all  $z$ such that  
\begin{align}   \label{eksperiment}
 e z = z     \;\;\;  \mbox{ and }   \;\;\; 
   \forall x y    \;     [ \  x (yz) = (xy) z  \  ]     \; .
   \end{align} 
   
 \begin{quote} {\bf (Claim)} \;\;\;\;\;\;
 We have $e\in A$. Furthermore, for any $z_0,z_1\in A$, we have $z_0 z_1\in A$  and $\llf{ z_0 }{ z_1}\in A$.
 \end{quote}  
 
 We prove the claim. By $\fintnavn_3$, we have
  $e e = e$ and    
   $x (ye) = xy = (xy) e$.  
 This proves that $e\in A$.  Assume $z_0,z_1\in A$.  By (\ref{eksperiment}), we have
  $e(z_0z_1) = (ez_0)z_1= z_0z_1$ and
  $$
  x(y(z_0z_1))\; =  \; x((yz_0)z_1) \; =  \; (x(yz_0))z_1 \;  = \;  ((xy)z_0)z_1  \;  =  \; (xy)(z_0z_1)\; .
  $$
 This  proves that $z_0z_1\in A$. Finally, we prove that $\llf{z_0}{z_1}\in A$. 
 By $  \fintnavn_4 $ and   (\ref{eksperiment}), we have 
$ e (\llf{z_0}{    z_1} )  =  \llf{(e z_0  )}{  z_1}   = \llf{z_0}{    z_1}     $. 
Furthermore, we have
  \[\renewcommand{\arraystretch}{1.3}   \begin{array}{lclll}
\;\;\;\;\;\;\;\;\;\;\;\;\;\;\;\;\;\;\;\; x(y(\llf{z_0}{z_1}))  &\; = \;& x(\llf{(yz_0)}{z_1}) &\;\;\;\;\; \;\;\;\;\;&  \ccom{(by $\fintnavn_4$)} \\
   & = & \llf{(x(yz_0))}{z_1}  & &   \ccom{(by $\fintnavn_4$)} \\
    & = & (\llf{(xy)z_0)}{z_1}  & &   \ccom{(by $z_0\in A$ and (\ref{eksperiment}))} \\
       & = & (xy)(\llf{z_0}{z_1})\; .  & &   \ccom{(by $\fintnavn_4$)}\\
\end{array}   \renewcommand{\arraystretch}{1.0} \]
This concludes the proof of the claim.

Let  $  v \sqsubseteq_{A}  w  $  be  shorthand for    $  \exists t  \in A   \,    [         w  =  v t      ]  $.  
Observe that $ \sqsubseteq_A $ is reflexive since $ e \in A $. 
Observe   also that $ \sqsubseteq_A $ is transitive.  
Indeed, assume $  x  \sqsubseteq_A  y $  and   $ y  \sqsubseteq_A z $. 
Then, $  z = y t $ and $ y = xs  $ for some $ s, t \in A $. 
Hence, $ z = (xs) t  = x (st) $ where the second  equality holds since $ t \in A $. 
Since $A$ is closed under $  \circ $, we have $ st \in A $. 
Thus, $ x  \sqsubseteq_A z $.

Let  $J$ denote the class of all $w$ such that   for any $  v \sqsubseteq_{ A }  w  $ 
\begin{multline} \label{heine}
   \forall x z  \;   \forall y \in A      \;   \Big[   \      xy = z v   
   \rightarrow 
\\
 \exists u  \in A    \;      \big[     \     
  \big(   \     ( x = zu    \;    \wedge   \;    uy = v  )    \;   \vee  \;    ( z= xu    \;   \wedge   \;      uv = y   \    )    \     \big)    \        \big]     
\        \Big] 
\   . 
\end{multline}
Since $ \sqsubseteq_A $ is reflexive, $  J  \subseteq   A  $.  
Since $  \sqsubseteq_A $ is transitive, $ J$ is downward closed under  $  \sqsubseteq_A $, that is, 
$ v  \in J $  if $  v  \sqsubseteq_A  w $  and $ w  \in J $. 
This implies that $J$ is also   downward closed under $  \preceq $ since  $ x \preceq y $ implies $ x \sqsubseteq_A y $
 as $ J  \subseteq  A $.  
Downward closure under $  \sqsubseteq_A $ ensures that $  \mathsf{Seq}^*_5 $ holds restricted to $J$. 
The other axioms of $  \mathsf{Seq}^* $ hold restricted to $J$ since they are universal.

We need to prove the next claim.

 \begin{quote} {\bf (Claim)} \;\;\;\;\;\;
 We have $e\in J$. Furthermore, for any $w_0,w_1\in J$, we have $w_0 w_1\in J$  and $\llf{ w_0 }{ w_1}\in J$.
 \end{quote}

It is easy to see that we have $ e  \in J $ since, by $  \fintnavn_1 $,  $  \fintnavn_4 $  and $  \fintnavn_5$,  
we have $ x  \sqsubseteq_A e $ if and only if $ x = e $.

Next we show that $J$ is closed under $\llf{ \,  }{ \, }  $. 
Assume $  w_0, w_1 \in J $ and $ v  \sqsubseteq_A  \llf{w_0 }{w_1} $. 
We need to show that $v$ satisfies (\ref{heine}). 
Now, $ v  \sqsubseteq_A   \llf{w_0 }{w_1} $ is shorthand for $\exists t\in A[vt = \llf{w_0 }{w_1}]$.
By $\fintnavn_5$ and $\fintnavn_3$ we have the two cases
\begin{itemize}
\item Case (i): $v=  \llf{w_0 }{w_1}$ 
\item Case (ii): $vt   = \llf{w_0 }{w_1}$ where $t=\llf{t_0 }{t_1}$ for some $t_0,t_1$.
\end{itemize}

First we deal with case (i). We assume 
$xy = z(\llf{w_0 }{w_1})$ (the antecedent in (\ref{heine})). We will prove that
\begin{align} \label{henrich}
\exists u  \in A    \;      \big[     \     
  \big(   \     ( x = zu    \;    \wedge   \;    uy = \llf{w_0 }{w_1}  )    \;   \vee  \;    ( z= xu    \;   \wedge   \;      u(\llf{w_0 }{w_1}) = y   \    )    \     \big)    \        \big]   
\end{align}
(the succedent in (\ref{heine})) holds. By $\fintnavn_5$ and $\fintnavn_3$, we split the proof 
into the (sub)cases 
\begin{itemize}
\item Case (i.a): $y=e$ and $x = z(\llf{w_0 }{w_1})$  
\item Case (i.b): $x(\llf{y_0 }{y_1}) = z(\llf{w_0 }{w_1})$  for some $y_0,y_1$.
\end{itemize}
In case (i.a), we have $x = z(\llf{w_0 }{w_1}) \wedge (\llf{w_0 }{w_1})y = (\llf{w_0 }{w_1})$, and thus (\ref{henrich}) holds (the first
disjunct holds). We turn to case (i.b). By $\fintnavn_4$, we have $\llf{(xy_0 )}{y_1} = \llf{(zw_0 )}{w_1}$. By $\fintnavn_2$, we have
$xy_0 = zw_0$ and $y_1 = w_1$. Morover, we have assumed $w_0\in J$, and we have $y_0\in A$ since $y\in A$. Thus,
by (\ref{heine}) and $xy_0 = zw_0$, we have
\begin{align} \label{wipes}
 \exists u \in A   \;      \big[     \       
   ( x = zu    \;    \wedge   \;    uy_0 = w_0   )    \;   \vee   \;    ( z= xu    \;   \wedge   \;      uw_0  = y_0    \    )        \        \big]     
   \      .
  \end{align}
From (\ref{wipes}), $\fintnavn_4$ and the fact that $y_1 = w_1$, we can conclude that $(\ref{henrich})$ holds.
This proves that the claim holds in case (i).

We turn to case (ii), that is the case when  $vt= v(\llf{t_0 }{t_1})   = \llf{w_0 }{w_1}$. First we prove that we have $t_0\in A$.

By $  \mathsf{Seq}_4 $ and the fact that $ A  \ni  t =   \llf{t_0}{t_1}  $, we have 
$ \llf{ (e t_0)}{   t_1 } =    \llf{t_0}{t_1}      $   and   $  \llf{  (   (ab) t_0 )}{   t_1}  
   =   \llf{ ( a (bt_0 ) ) }{   t_1}    $, 
which gives $ e t_0 = t_0 $  and  $ (ab) t_0 = a ( bt_0 ) $ by $  \mathsf{Seq}_2 $. 
This proves that $ t_0 \in A $. 

By $  \mathsf{Seq}_4$,  we get $\llf{ w_0}{  w_1} = vt = \llf{(v t_0) }{ t_1} $. 
By $  \mathsf{Seq}_2 $, $ w_0 = v t_0 $ which means $ v \sqsubseteq_A w_0 $. 
It follows that    $v$ satisfies (\ref{heine})  since $ w_0 \in J $. This concludes the proof of case (ii), and thus we have proved that $J$ is
closed under $\llf{ \, }{ \, }$.

Finally, we   show that $J$ is  closed under $\circ$. 
So, let $v_0, v_1 \in J$ and suppose $w  \sqsubseteq_{A} v_0 v_1$. 
We need to show that $w $ satisfies (\ref{heine}). 
From $w \sqsubseteq_{A} v_0 v_1$ and the definition of $ \sqsubseteq_{A}  $,   
we  have  $v_0 v_1  =  wx $  for some $x\in A$.
Since  $v_1 \in J $ and $  \sqsubseteq_A $ is reflexive,   we have one of the following cases for some $s \in A$:
(1)    $v_0= ws  $   and   $ sv_1= x$, 
(2)   $w = v_0 s   $   and   $ s x = v_1$.
In the first case,  $v_0= w s $ and $s\in A$ implies $w \sqsubseteq_{A} v_0$. 
Hence, $ w$ satisfies (\ref{heine}) since  $ v_0 \in J $. 
In the second case,  $s x = v_1$ and $x\in A$ implies $s \sqsubseteq_{A} v_1$. 
Since $v_1 \in J$ and $J$ is downward closed under $  \sqsubseteq_A$, we  get   $s \in J$. 
Let $ w_0 := v_0 $ and $ w_1 := u $. 
We need  to show that $ w_0 w_1 $ satisfies (\ref{heine}). 
So,   assume  $xy = z (w_{0} w_{1})$   and  $ y \in A $. 
Since $w_1 \in A$, we  get  $xy = (z w_{0} ) w_{1} $.
Since  $w_{1} \in J$, we have one of the following cases for some  $  u \in A $: 
(I) $x = (zw_{0} ) u   $  and   $    u y = w_{1} $, 
(II) $zw_{0} = xu   $   and    $   uw_{1} = y  $. 
In case (I),  $x = (zw_{0} ) u $ implies $x = z (w_{0}  u ) $ since $u \in A$. 
Since $A$ is closed under $\circ$, we have $w_{0}  u \in A$. 
Then,  $x = z (w_{0}  u )  $   and   $ w_0 w_1 = w_0 ( uy) = (w_0 u) y$,  where we have used the fact that $y\in A$. 
So, in this case  we   get  that  $w_0 w_1 $ satisfies (\ref{heine}).

We  consider (II). 
It follows from $zw_{0} = xu $, $u \in A$ and $w_0 \in J$ that one of the following holds for some $u^{\prime} \in A$:
(IIa)  $z= x u^{\prime}   $  and   $ u^{\prime} w_{0} = u$,  
(IIb)  $x= z u^{\prime}    $    and   $  u^{\prime} u = w_{0}$. 
From  (IIa) we get $z= x u^{\prime}   $  and   $  y= u w_1  = ( u^{\prime} w_{0} )  w_{1} = u^{\prime}  (w_{0}   w_{1} )  $, 
 where we have used the fact that $w_1 \in A$. 
From  (IIb) we get $x= z u^{\prime}   $  and   $  
w_{0} w_1 =  ( u^{\prime} u  ) w_1 = u^{\prime} (u w_1) =  u^{\prime}  y $, 
 where we have used the fact that $w_1 \in A$ and $u w_1 = y$.
Hence, in case of (II) we also get  that $ w_0 w_1 $ satisfies (\ref{heine}). 
Hence, $ v_0 v_1 \in J $. 
Thus, $J$ is closed under $\circ$, and thus we conclude that the claim holds.

To complete the proof of our lemma, we need to prove yet another claim.

 \begin{quote} {\bf (Claim)} \;\;\;\;\;\; We have  $  \forall w   [   \llf{e}{w}  \in  A     ]    $  and   
 $  \forall w   [     \llf{e}{w}  \in  J      ]    $. 
 \end{quote}

It can be checked that    $  \forall w   [     \llf{e}{w}  \in  A       ]    $  and  that   
  $ v   \sqsubseteq_A  \llf{e}{w} $ if and only if $ v = e  $ or $ v = \llf{e}{w} $. 
Hence, it suffices to show that $ \llf{e}{w}$ satisfies (\ref{heine}). 
So, assume $ y \in A $ and $ xy = z (\llf{e}{w} ) $. 
If $ y = e $, then $  x = z (\llf{e}{w}) $ and $ (\llf{e}{w} ) e = \llf{e}{w} $. 
Otherwise, by $  \mathsf{Seq}_5 $, $ y = \llf{y_0} { y_1 }$ for some $ y_0, y_1 $. 
Since $ y \in A $, we have $ y_0 \in A $. 
By $  \mathsf{Seq}_4 $ and $  \mathsf{Seq}_2 $, 
from  $ xy = z (\llf{e}{w} ) $ we get $ z =  x y_0  $  and $ y_1 = w $, which implies  $  y_0  (\llf{e}{w} )  = y $. 
Hence, $ \llf{e}{w} $ satisfies (\ref{heine}). 
Thus, $ \llf{e}{w} \in J $ and the claim holds. 
\qed
\end{proof}

\subsection{$  \mathsf{Seq} $ Plus Extras  }

In this section we introduce indexed sequences and show that we can reason about them in the usual way. 
Furthermore, we introduce an extension  $  \mathsf{Seq}^+$ of $  \mathsf{Seq} $, which is convenient  for working with indexed sequences, 
and show that it is mutually interpretable with $  \mathsf{Seq} $. 
The theory   $  \mathsf{Seq}^+$   is  given by the axioms of $  \mathsf{Seq} $, the axioms of   $  \mathsf{Seq}^* $ and an axiom $ \mathsf{Seq}^+_c$ which says that $ \circ $ is left and right cancellative, see Figure \ref{fintnavnplussaksiomer}.

\begin{figure}[t] 
$$\begin{array}{lll}
 \fintnavn_1 & \; & \forall x y [  \  \llf{x}{y} \neq \llc   \   ]  \\
 \fintnavn_2 & \; &  \forall x_1 x_2 y_1 y_2 [  \  \llf{x_1}{x_2} = \llf{y_1}{y_2} \; \rightarrow \; (  \  x_1 =y_1  \, \wedge \, x_2 = y_2 \  )       \  ]  \\
 \fintnavn_3   & \; & \forall x  [  \  \llapp{x}{\llc} =  x \   ] \\
 \fintnavn_4   & \; & \forall x y z   [  \  \llapp{x}{(\llf{y}{z})}   = \llf{(\llapp{z}{y})}{z} \     ] \\
 \fintnavn_5  & \; & \forall x    [  \  x= \llc \; \vee \; \exists y z [ \    x= \llf{y}{z}      \ ] \     ] \\
 \mathsf{Seq}^*_3  &   \;  &   
\forall x      [     \     xe = x  \;  \wedge  \;   ex = x    \      ]  \\
\mathsf{Seq}^*_5  &  \;  &   
\forall x y z w       \big[ \  xy = zw   \;\;  \leftrightarrow   \;\;
  \exists u [ \, ( z=xu   \,  \wedge   \,      uw = y )   \,   \vee   \,    (    x = zu     \,  \wedge   \,      uy = w     )   \,   ]     \         \big] \\
  \mathsf{Seq}^+_c  &  \;  & \forall x  y z     [    \   (     xy = x z  \,   \vee  \,   y x = z x       )  \; \rightarrow \;  y = z    \      ] 
\end{array}
$$
\caption{The non-logical axioms of $\fintnavn^+$.
\label{fintnavnplussaksiomer}}
\end{figure}

Recall that we have shown that there exists a domain  $  \mathcal{J} $ such that the axioms of $  \mathsf{Seq}^* $ hold restricted to $ \mathcal{J} $
and $  \forall w   [      \llf{e}{w}  \in    \mathcal{J}       ]    $.
Let   $  x  \preceq y $ be shorthand for    $   \exists z  \in \mathcal{J}     [        y = xz      ]   $. 
Then, $ \mathcal{J} $ is downward closed under the reflexive-transitive relation $ \preceq  $.  
Observe though that $  \mathsf{Seq}_5 $ may not hold restricted to  $  \mathcal{J} $.
In this section, we work in $  \mathsf{Seq} $ but restrict our attention to objects in $  \mathcal{J} $, 
which means that we have the resources of $  \mathsf{Seq}^* $ at our disposal. 
We write $ x \prec y $ for $  x  \preceq y     \;  \wedge  \;     x  \neq  y  $.

We start by defining  a  totally ordered   class $  \mathsf{N}   $  of number-like objects  
and operations of successor $  \mathrm{S} $ and addition  $+$ on $  \mathsf{N} $
 that provably in $  \mathsf{Seq}  $   satisfies the first five axioms of Robinson's $  \mathsf{Q} $: 
 (1) $  \mathrm{S} $ is one-to-one, 
 (2) $0$ is not in the range of $  \mathrm{S} $, 
 (3) any non-zero element has a predecessor, 
 (4)-(5) primitive recursive equations for $+$. 
Let $ 0 := e $,  $ \mathrm{S} x  :=  \llf{x}{  e} $ and $ x + y := x  y $. 
By $  \mathsf{Seq}_1 $, $  \mathsf{Seq}_3 $, $  \mathsf{Seq}_4 $ and $  \mathsf{Seq}_5 $,
$  \forall x  \;   [    \   x    \preceq   0 \leftrightarrow  x = 0   \     ]   $  
and 
$  \forall x, y \;   [    \       (     \     x = \mathrm{S} y  \;   \vee  \;   x   \preceq  y    \      )    \rightarrow    x  \preceq    \mathrm{S} y   \    ]   $
hold restricted to $  \mathcal{J} $. 
We   have 
$  \forall x, y \;   [    \   x  \preceq    \mathrm{S} y      \rightarrow   (     \     x = \mathrm{S} y  \;   \vee  \;   x   \preceq  y    \      )      \    ]   $
by  $  \mathsf{Seq}_5 $, $  \mathsf{Seq}_4 $ and the closure properties of  $  \mathcal{J} $. 
By $  \mathsf{Seq}_2 $ and  $  \mathsf{Seq}_1 $,  $  \mathrm{S} $ is one-to-one and $0$ is not in the range of $  \mathrm{S} $.
By  $  \mathsf{Seq}_3^*$ and $  \mathsf{Seq}_4$,  the primitive recursive equations for  $+$ hold. 
By the right-left implication of  $  \mathsf{Seq}_5^* $,  $  \forall x y z  \;  [    \    (x+y)+z = x + (y+z)    \      ]  $     holds  restricted to $  \mathcal{J} $. 
Let $ [0, x ] $  denote the class of  of all $ u $ such that $ u  \preceq x $. 
Observe that $ \forall x   [   0  \preceq  x   ] $ holds since  $e$ is an identity element with respect to $  \circ $  on $  \mathcal{J} $.

Let $ I $ consist of all $ w  \in \mathcal{J}$  such that for all $  x   \preceq  w $ the following hold: 
\begin{itemize}
\item (1)  $  \left(  \,  [0, x ]  \,  ,  \,   \preceq  \,   \right) $ is a total order
\item (2) $x= 0 $ or there exists $y$ such that $ x =  \mathrm{S} y $
\item (3) if $ u \prec x $, then $  \mathrm{S} u   \preceq   x  $,
\item (4)  if $ y \in \mathcal{J} $ and  $  \forall u  \prec   y     [       \mathrm{S} u  \preceq  y      ]  $
 then  $ x  \preceq  y $ or  $  y  \preceq  x $
 \item (5) for any $ u  \in \mathcal{J} $,     we have    $  \mathrm{S} (  x + u  ) =  \mathrm{S} x + u $
\item (6)   for any $ u, v \in \mathcal{J} $, if $ u+x = v + x $, then $ u = v $. 
\end{itemize}
It can be checked that $I $ contains $0$,  is closed under  $ \mathrm{S}$ and is downward closed under $  \preceq $, that is, 
if $ y  \preceq  x $ and $ x \in I $, then $ y \in   I $. 
Let $ J $ consist of all $ w \in I $ such that for any $ x  \preceq w $ and  any $ u \in I $ we have $ u + x = x + u $. 
It can be checked that $J $ contains $0$,  is closed under  $ \mathrm{S}$ and is downward closed under $  \preceq $. 
We restrict   $I$ to a subclass  $  \mathsf{N} $ that is also closed under addition (this uses associativity of $+$):  
$  \mathsf{N} $ consist of all $ x \in J $ such that $  \forall y \in J     [       y  +  x \in J       ]    $. 
 Let $ x   \leq  y $ be shorthand for $  \exists z \in  \mathsf{N}       [      x +  z = y       ] $. 
We write $ x < y $ for $  x \leq y    \wedge   x  \neq  y $.

\begin{lemma}[$  \mathsf{Seq}  $]
The first five axioms of Robinson's $  \mathsf{Q} $ hold restricted to $  \mathsf{N} $, and
 $( \mathsf{N} , \leq ) $ is a linear order with least element $0$, moreover      
$$  \forall x, y  \in \mathsf{N} \;   [    \   x  \leq \mathrm{S} y   \leftrightarrow     (     \     x = \mathrm{S} y  \;   \vee  \;   x  \leq y    \      )      \    ]   \; .$$
Furthermore, addition on $  \mathsf{N} $  is associative,  commutative and cancellative, 
and  $  \forall x \in \mathsf{N}    \forall  y  \leq x    [       y  \in  \mathsf{N}       ]     $. 

\end{lemma}

To improve the readability, we will occasionally  
use $ (x , y ) $ as alternative notation for $  \llf{x}{y} $ (we will use this notation when $ x  \in  \mathsf{N}$ 
and it is natural to regard $x$ an index pointing to $y$).

We will now define the class $\mathcal{S}_0 $ of \emph{indexed  sequences}.
Let  $ f  \in  \mathcal{S}_0 $ iff $f  \in \mathcal{J} $ and there exist $ m , n \in \mathsf{N} $ such that
\begin{itemize}
\item (0) $ m \leq n $
\item (1) for each  $ m \leq  k \leq n $ there exists  $u  \in \mathcal{J} $ such that   
$ f =  \big( \llf{s}{(k,u)}  \big) t $ for some $s, t   \in \mathcal{J} $
\item (2) if $ f  =  ( \llf{s}{w}) t $ for some $ s, t, w   \in \mathcal{J} $, 
then there exist $  m \leq k \leq n $ and $ u  \in  \mathcal{J}  $ such that $ w = (k, u ) $
\item (3) if $  f =  \big( \llf{s}{(k, u)}   \big) t  $ and $ s \neq e $, then $ m < k $ and  there exist $ \ell  ,  v ,  s_0   \in \mathcal{J}$  such that 
$ k = \mathrm{S}  \ell $ and  $ s = \llf{s_0}{ ( \ell, v )} $
\item (4) if $  f =  \big( \llf{s}{(k, u)}  \big) t  $, then  there is no $ k  \leq  k^{ \prime }  \in \mathsf{N} $ and $ u^{  \prime } , s_0 , t_0  \in   \mathcal{J} $ such that 
$ s  =  (  \llf{s_0}{ ( k^{  \prime } ,   u^{  \prime } )}    )   t_0   $. 
\end{itemize}


If $f\in  \mathcal{J}$ and (0), (1),  (2), (3), (4) hold,
we will call $ [m, n ] $ the domain of $f$ and use the suggestive notation $ f: [m, n ]  \to \mathcal{J}  $, furthermore,
if $ m \leq k \leq n $, we write $ f(k) $ for the unique element $u\in  \mathcal{J}$ 
such that  $ f =  \big( \llf{s}{  (k,  u )}  \big) t $ for some $s, t   \in \mathcal{J} $.
To see that $u$ indeed is unique, assume 
$ f = \big( \llf{s_0}{ (k,  u )}   \big)  t_0  $  and   $ f = \big( \llf{s_1}{ (k,  v )}   \big)  t_1  $.
Since the editor axiom  holds restricted to $  \mathcal{J} $ 
(as  $  \mathsf{Seq} \vdash   ( \mathsf{Seq}_5^{ *} )^{  \mathcal{J}  }  $), 
we have one of the following:
(i) $ \llf{s_0}{ (k,  u )}   =  \llf{s_1}{ (k,  v )} $, 
(ii)  there exists $ w \in \mathcal{J} \setminus  \lbrace e  \rbrace  $ such that  
$ \llf{s_0}{(k,  u )}     =  \big(        \llf{s_1}{(k,  v )}      \big)   w $, 
(iii) there exists $ w \in \mathcal{J}  \setminus  \lbrace e  \rbrace $ such that   
$  \llf{s_1}{(k,  v ) }    =  \big( \llf{s_0}{(k,  u ) }  \big)   w $.
In case of (i), we get $ u = v $ by  $  \mathsf{Seq}_2 $. 
We show that (ii) and (iii) yield contradictions. 
We consider (ii). 
By $  \mathsf{Seq}_5 $, there exist $ w_0, w_1 $ such that $ w = \llf{w_0}{w_1}  $. 
By $  \mathsf{Seq}_3 $ and $  \mathsf{Seq}_4 $, we get $ w = w_0 ( \llf{e}{w_1} ) $. 
By the closure properties of $  \mathcal{J} $, we have   $  \llf{e}{w_1}  \in   \mathcal{J} $. 
By $  \mathsf{Seq}_2 $ and  $  \mathsf{Seq}_4 $, 
 from the equality    $ \llf{s_0}{(k,  u )}   =   \big(\llf{s_1}{(k,  v )}   \big)   \big(  \llf{w_0}{w_1}      \big)  $  we get 
$ s_0  =   \big( \llf{s_1}{(k,  v )}   \big)   w_0 $, which contradicts  (4). 
By similar reasoning, (iii) yields a contradiction.
Thus, we conclude that $ u = v $.

We call an element $f$ an \emph{indexed  sequence}  if and only if $ f  \in  \mathcal{S}_0 $. 
The following lemma can be proved  using, in particular, the fact that  the editor axiom holds restricted to $  \mathcal{J} $.

\begin{lemma}
Consider an indexed sequence $ f  : [ m , n ]  \to \mathcal{J} $ and  an index $ m \leq k < n $. 
Then, there exist  two unique indexed sequence $ g :  [ m, k   ]  \to \mathcal{J} $ and  $ h :  [ k+1, n ]  \to \mathcal{J} $  
such that $ f = g \circ h $. 
Furthermore, $ g = \llf{s} { (k, f(k) )} $ and $ h = \llf{t}{    (n, f(n) )} $ for some $ s, t  \in  \mathcal{J} $. 
\end{lemma}

From now on, we refer to the unique indexed sequences $g$ and  $h$ of the preceding lemma as the restrictions of $f$ to 
$ [m, k ] $ and  $ [k+1, n] $, respectively.

\begin{theorem}  \label{MainTheoremIndexedSequences}

There exists a subdomain $  \mathsf{N}^{ \star } $ of $  \mathsf{N} $ such that 
\begin{itemize}

\item \textup{(I)}  for any  element $x \in  \mathcal{J}$ and any $   n \in   \mathsf{N}^{ \star }  $
 there exists a unique  indexed sequence  $ \frac{ x }{ [0, n ] }  : [0, n ]  \to \mathcal{J}  $   such that 
$  \frac{ x }{ [0, n ] }  (j)  = x $ for all $  0   \leq j \leq n $

\item \textup{(II)}  for   any $ k, n \in  \mathsf{N}^{ \star }  $  and any     indexed sequence  $ f : [0, n ]  \to \mathcal{J}  $,
 there exists   a    unique  indexed sequence
$  f^{ 0 }_{ k }    : [k, k+n ]  \to \mathcal{J}  $
 such that $ f^{ 0 }_{ k }   (k+j ) = f(j)  $      for all $ 0 \leq j \leq n $

\item \textup{(II)}   for   any $ n \in  \mathsf{N}^{ \star }  $  and any   two  indexed sequences  $ f, g : [0, n ]  \to \mathcal{J}  $,
 there exist two   unique  indexed sequences  
$ f \oplus g  : [0, n ]  \to \mathcal{J}  $ and  $ f \ominus g   : [0, n ]  \to \mathcal{J}  $
 such that $ f \oplus g    (j ) = f(j) + g(j) $ and  $ f \ominus g  (j) = f(j) \dotdiv g(j) $     for all $ 0 \leq j \leq n $

\item \textup{(III)}   for   any $ n, m  \in  \mathsf{N}^{ \star }  $  and any   two  indexed sequences  $ f : [0, n ]  \to \mathcal{J}  $ and  $ g: [0, m ]  \to \mathcal{J}  $
there exists a unique   indexed sequence $ f ^{ \frown }   g  : [0, n+m +1  ]  \to \mathcal{J}  $ 
 such that $  f ^{ \frown }   g = f \circ g ^{ 0 }_{ n+1 }   $   and  
 \[
 f  ^{ \frown } g    (j ) = \begin{cases}
  f(j) &   \mbox{ if }  0 \leq j \leq n
  \\
  g(j \dotdiv (n+1)  )   &   \mbox{ if }  n+1  \leq j \leq  n+m+1 
  \end{cases}
\]

\item \textup{(IV)}  for   any $ n \in  \mathsf{N}^{ \star }  $, any $ 0 \leq k  < n $  and any     indexed sequence  $ f : [0, n ]  \to \mathcal{J}  $,
 there exist     unique  indexed sequences  $  g  : [0,   k  ]    \to \mathcal{J}  $  and $  h  : [0,   n \dotdiv  (k+1)   ]    \to \mathcal{J}  $ 
 such that $ f = g^{ \frown }  h  $

\item \textup{(V)}  for   any $ n, m, k   \in  \mathsf{N}^{ \star }  $  and any   three   indexed sequences  $ f : [0, n ]  \to \mathcal{J}  $,   $ g: [0, m ]  \to \mathcal{J}  $
and $ h: [0, k ]  \to \mathcal{J}  $, we have $ ( f  ^{ \frown } g )  ^{ \frown } h = f  ^{ \frown }  ( g  ^{ \frown }  h ) $

\item  \textup{(VI)}  for   any $ n, m, k, \ell    \in  \mathsf{N}^{ \star }  $  and any   four    indexed sequences  $ f : [0, n ]  \to \mathcal{J}  $,   $ g: [0, m ]  \to \mathcal{J}  $, 
 $ p : [0, k ]  \to \mathcal{J}  $ and $  q  : [0,  \ell  ]  \to \mathcal{J}   $  such that 
 $    f  ^{ \frown }  g = p   ^{ \frown }    q $, we have one of the following:  
 \begin{itemize}

\item $ f = p $    and $ g = q $

\item there exist a unique   $ r  \in  \mathsf{N}^{ \star }     $ and   a unique  indexed sequence  $ h  :  [0, r ]  \to \mathcal{J}  $ such that 
$ f = p ^{ \frown } h $  and $ h  ^{ \frown } g = q $

\item there exist unique $ r \in    \mathsf{N}^{ \star }  $ and   a unique   indexed sequence  $ h  :  [0, r ]  \to \mathcal{J}  $ such that 
$ p = f ^{ \frown } h $  and $ h  ^{ \frown } q = g $.

 \end{itemize}

\end{itemize}

\end{theorem}

\begin{proof}
Let $K$ consist of all $ n^{ \star }   \in \mathsf{N} $ such that  (I)-(VI) hold when we replace each occurrence  of 
$ i  \in \mathsf{N}^{ \star } $ with $ i    \leq  n^{ \star }  $. 
Clearly, $K$ is downward closed under $  \leq $. 
It can be checked that $ K$ is closed under $0 $ and $  \mathrm{S} $. 
We settle with verifying that if $  n^{ \star }  \in K $, then (VI) holds when we replace occurrences of  $ i  \in \mathsf{N}^{ \star } $ with   $ i \leq   n^{ \star }  +1 $. 
So, assume  $ n, m, k, \ell   \leq   n^{ \star }  +1 $ and   $    f  ^{ \frown }  g = p   ^{ \frown }    q $ where 
\[
f : [0, n ]  \to \mathcal{J}   \, ,   \    \     \    g: [0, m ]  \to \mathcal{J}  \, ,   \    \     \   
   p : [0, k ]  \to \mathcal{J}    \, ,   \    \     \       q  : [0,  \ell  ]  \to \mathcal{J}   
   \       . 
\]
We have three cases: (i) $ m , \ell   > 0 $, (ii) $ m = 0 $, (iii) $\ell  = 0 $. 
We consider (i), the other cases are handled by similar reasoning. 
Let $g_0$ and $ q_0$ denote the restriction of $g$ and $q$ to  $ [0, m \dotdiv 1 ] $ and $ [0, \ell \dotdiv 1 ] $, respectively. 
We have 
\[
 g \; = \; \llf{g_0}{  (m, g(m) )} \;  = \;  g_0 ^{ \frown } \frac{ g(m) }{ [0, 0 ] }
\    \    \mbox{  and   }      \      \    
  q \; =  \; \llf{q_0} { ( \ell , q( \ell )    )} \;    = \;  q_0 ^{ \frown } \frac{ q( \ell ) }{ [0, 0 ] } 
  \       .
  \]
By the editor axiom restricted to $  \mathcal{J} $, from  
\begin{align*}
f \circ (g_0)^{ 0 }_{ n+1} \circ  \frac{ g(m) }{ [n+m+1, n+m+1 ] }    
&=
  f  ^{ \frown }  g 
  \\
  &= p   ^{ \frown }    q 
  \\
&=
 p \circ  (q_0)^{ 0 }_{ k+1} \circ  \frac{ q( \ell ) }{ [k+ \ell + 1, k+ \ell +1 ] }   
\end{align*}
we  get $    f  ^{ \frown }  g_0  = p   ^{ \frown }    q_0 $   and  $ g(m)   =   q( \ell )   $. 
Since  $  n^{ \star }  \in K $, from  $    f  ^{ \frown }  g_0  = p   ^{ \frown }    q_0 $ we have that one of the following holds: 
(1)  $ f = p $    and $ g_0  = q_0 $, 
(2)    there exist a unique   $ r  \leq   n^{ \star }   $ and   a unique  indexed sequence  $ h  :  [0, r ]  \to \mathcal{J}  $ such that 
$ f = p ^{ \frown } h $  and $ h  ^{ \frown } g_0  = q_0 $, 
(3)   there exist a unique  $ r      \leq   n^{ \star }   $ and   a unique   indexed sequence  $ h  :  [0, r ]  \to \mathcal{J}  $ such that 
$ p = f ^{ \frown } h $  and $ h  ^{ \frown } q_0  = g_0  $. 
From this and the fact that $ g(m)   =   q( \ell )   $, we get that 
one of the following holds: 
(4)  $ f = p $    and $ g  = q $, 
(5)    there exist a unique   $ r  \leq   n^{ \star }   $ and   a unique  indexed sequence  $ h  :  [0, r ]  \to \mathcal{J}  $ such that 
$ f = p ^{ \frown } h $  and $ h  ^{ \frown } g  = q $, 
(6)   there exist  a unique $ r     \leq   n^{ \star }   $ and   a unique   indexed sequence  $ h  :  [0, r ]  \to \mathcal{J}  $ such that 
$ p = f ^{ \frown } h $  and $ h  ^{ \frown } q  = g  $.

We obtain  $  \mathsf{N}^{ \star } $ by restricting    $K$ to a subclass  which is also closed under addition: 
$  \mathsf{N}^{  \star } $ consist of all $ x   $ such that $  \forall y \in K  \;   [    \    y + x \in K    \     ]    $. 
It can be checked that $  \mathsf{N}^{  \star } $ is closed under $0, \mathrm{S} , + $ and is downward closed under $  \leq  $. 
\qed
\end{proof}

We interpret $  \mathsf{Seq}^+ $ in $  \mathsf{Seq} $ as follows: 
\begin{itemize}

\item The domain consists  of  $e$   and all indexed  sequences    $ f : [0, n ]  \to  \mathcal{J} $
 where $ n \in  \mathsf{N}^{  \star }   $
and $  f(j) \in  \mathcal{J} $ for all $  0 \leq j  \leq  n  $.

\item Let $e^{  \tau }  =  e $.

\item  If $ f = e $, let     $ f  \circ^{  \tau }   g  = g $. 
If $ g = e $,   let     $ f  \circ^{  \tau }   g  = f $. 
For  $ f : [0, n ]   \to    \mathcal{J} $ and    $ g : [0, m ]   \to    \mathcal{J} $, 
 let    $  f  \circ^{  \tau }   g   =  f  ^{ \frown }   g  $.

\item If $ f = e $, let   $ f  \vdash^{  \tau }   g   = \frac{g}{ [0, 0 ] }$.
For  $ f : [0, n ]   \to    \mathcal{J} $, 
 let     $ f  \vdash^{  \tau }   g = f ^{ \frown }  \frac{g}{ [0, 0 ] }  $.

\end{itemize}
It follows from the preceding theorem that  this defines an interpretation  of $  \mathsf{Seq}^+ $ in $  \mathsf{Seq} $. 
We settle with  verifying that the translation of  the left  cancellation law is 
 a  theorem of     $ \mathsf{Seq} $.
Assume $ f \circ^{ \tau }  g = f  \circ^{ \tau }  h $. 
Since $ e$ is an identity element, we may assume  that  $ f, g, h  \neq  e  $. 
In particular, $  f : [0, n ]  \to  \mathcal{J} $ for some $ n \in \mathsf{N}^{ \star } $. 
Then, $ f^{ \frown } g = f ^{ \frown } h $. 
By Clause (VI) of the preceding theorem, we have one of the following cases: 
(i)  $ g= h $, 
(ii) there exists an indexed sequence  $p :  [0,  \ell   ]  \to  \mathcal{J} $ such that $ f = f^{ \frown } p $ and $ p ^{ \frown } h = g $, 
(iii)  there exists an indexed sequence  $p :  [0,  \ell   ]  \to  \mathcal{J} $  such that $ f = f^{ \frown } p $ and $ p ^{ \frown } g = h $.   
In case of (ii) and (iii), we get  $ n = n+ \ell + 1 $. 
Since addition on $  \mathsf{N}^{ \star } $ is cancellative and $0$ is an additive identity, we get $ 0 = \ell + 1 $, 
 which contradicts the fact that $0$ is not in the range of $  \mathrm{S} $. 
Thus,  $ g = h $.

\begin{theorem}
$  \mathsf{Seq} $ interprets $  \mathsf{Seq}^+ $. 
\end{theorem}

From now on, we work in  $  \mathsf{Seq}^+ $  and simply express  indexed sequences as $ f : [0, n ]  \to   \mathcal{V} $
and implicitly assume  $ n \in \mathsf{N}^{ \star } $
 ($ \mathcal{V} $ denotes the universe of an arbitrary model of  $  \mathsf{Seq}^+ $).

\section{$\fintnavn$ interprets the tree theory $\mathsf{T}$ }

\label{gronnspisser}

\begin{figure}[t] 
$$\begin{array}{lll}
 \mathsf{T_1} & \; & \forall x y [  \  \txllf{x}{y} \neq \txllc   \   ]  \\
 \mathsf{T_2}   & \; &  \forall x_1 x_2 y_1 y_2 [  \  \txllf{x_1}{x_2} = \txllf{y_1}{y_2} \; \rightarrow \; (  \  x_1 =y_1  \, \wedge \, x_2 = y_2 \  )       \  ]  \\
 \mathsf{T_3}   & \; & \forall x  [  \  x \txllr  \txllc \; \leftrightarrow \; x = \txllc \   ] \\
 \mathsf{T_4}   & \; & \forall x y z   [  \  x \txllr  \txllf{y}{z} \; \leftrightarrow \;
 (  \  x= \txllf{y}{z} \, \vee \, x \txllr  y \, \vee  \,  x \txllr  z \  )  \     ]\; .
\end{array}
$$
\caption{The axioms of $\mathsf{T}$ 
\label{figab}}
\end{figure}

Recall the theory $\mathsf{T}$  
over the language $ \mathcal{L}_{ \mathsf{T} } = \lbrace \txllc,  \txllf{\cdot}{\cdot} , \txllr   \rbrace $ 
where $\txllc$ is a constant symbol, $\txllf{\cdot}{\cdot}$ is a binary function symbol and $\txllr$ is a binary relation symbol. 
The axioms of $\mathsf{T}$   is given Figure \ref{figab}, and
the intended model   is a term model: 
The universe is the set of all variable-free $\mathcal{L}_{ \mathsf{T} } $-terms. 
Each term is interpreted as itself, and 
 $\sqsubseteq$ is interpreted as the subterm relation ($s $ is a subterm of $t$ iff $ s = t $ or $ t = \txllf{t_1}{t_2}$ and $s$ is a subterm of $t_1 $ or $t_2$).  This term model is obviously isomorphic to a model where the universe consists of (full) binary trees
 and where $\sqsubseteq$ is interpreted as the subtree relation. 
 We will interpret  $\mathsf{T}$ in $\fintnavn$.

\subsection{The ideas}

Variable-free $\mathcal{L}_{ \mathsf{T} }$-terms will be called {\em finite binary trees}, or just {\em binary trees},
or just {\em  trees}.
Let $\{a,b\}^*$ denote the set of all finite strings over the alphabet $\{a,b\}$. We define the set $\polish\subseteq \{a,b\}^*$
inductively by
\begin{itemize}
\item the string $a$ is in $\polish$
\item the string $b\alpha\beta$ is  in  $\polish$ if  $\alpha$ and $\beta$ are strings in $\polish$.
\end{itemize}

Note that strings in $\polish$ can be viewed as the finite binary trees written in Polish notation. 
For any $\alpha \in \{a,b\}^*$, let $(\#_a  \, \alpha)$ and $(\#_b \, \alpha)$ denote,  respectively, the number of occurrences of the alphabet symbols $a$ and $b$ in $\alpha$.
The  next lemma is very well known.

\begin{lemma} \label{mistetpaagulv}
$\alpha\in \polish$ if and only if
\begin{enumerate}
\item $(\#_b  \, \alpha) + 1 = (\#_a  \, \alpha )$ (and thus $(\#_a  \, \alpha) > (\#_b  \, \alpha)$), and
\item if $\beta$ is a strict prefix of   $\alpha$, then 
$(\#_a  \, \beta) \leq (\#_b  \, \beta)$.
\end{enumerate}
\end{lemma}

The characterisation
given by Lemma \ref{mistetpaagulv} can be used to prove unique readability of Polish notation for binary trees:
If $\alpha\in \polish$, then either $\alpha=a$, or there exist unique $\beta,\gamma\in \polish$ such that $\alpha=b\beta\gamma$.

Finite sequences over the natural numbers of a certain form will be called {\em snakes}.

\begin{definition}
The sequence $x_1, \ldots, x_n$  is a {\em snake} if $n=1$ and $x_1 = x_n = 0$.
The sequence $x_1, \ldots, x_n$  is a {\em snake} if 
\begin{itemize}
\item $x_1= 2$ and $x_n= 0$ (and thus $n>1$)
 \item $ x_i >0 $ for  $i=1, \ldots , n-1$   
\item $\vert x_i - x_{i+1}\vert = 1$ for $i=1, \ldots , n-1$.
\end{itemize}
\end{definition}

\paragraph{Examples:} The sequence $2,1,0$ is a snake. The sequence $2,1,2,1,2,1,0$ is a snake
and so is $2,1,2,3,4,5,4,3,2,1,2,3,4,3,2,1,0$. \qed

We will represent binary trees by snakes. Each tree will be represented by a unique snake (and each snake will represent some tree).
For any binary tree $t$, let $\widehat{t}$ denote the corresponding string in $\polish$, that is, $\widehat{\txllc}=a$
and $\widehat{\txllf{t_1}{t_2}}=b\widehat{t_1}\widehat{t_2}$. The binary tree $t$ will
be represented by the  snake $\tilde{t}$ given by
\begin{itemize}
 \item if $\widehat{t}=a$, let $\tilde{t}=0$ (that is, $\tilde{t}$
is a sequence of length one, and  0 is the one and only element of the sequence)
\item if  $\widehat{t}= \alpha_1 \alpha_2 \ldots \alpha_n $ where $\alpha_i\in \{ a,b\}$, let $\tilde{t}=(x_1, \ldots x_n)$ where
\begin{align} \label{calcium}
x_1=2 \;\;\; \mbox{  and } \;\;\;
x_{i+1} \; =  \; \begin{cases}
 x_i - 1& \mbox{if $\alpha_{i+1}=a$}  \\
  x_i + 1& \mbox{if $\alpha_{i+1}=b$.}  
\end{cases} 
\end{align}
for $i=1,\ldots, n-1$. 
\end{itemize}

Let us check that $\tilde{t}=(x_1, \ldots x_n)$ indeed is a snake. 
Let $\widehat{t}= \alpha_1 \alpha_2 \ldots \alpha_n $where $\alpha_i\in \{ a,b\}$. By (\ref{calcium}), we have
$$
(\#_b \; \alpha_1 \ldots \alpha_i)  \; + \; 1 - \;  (\#_a \; \alpha_1 \ldots \alpha_i) \;\;  =  \;\; x_i
$$
and thus, it follows from Lemma \ref{mistetpaagulv} that $\tilde{t}$ is a snake (for any tree $t$). 
It also follows from our discussion above and Lemma \ref{mistetpaagulv} that 
the mapping $t \mapsto \tilde{t}$ is a bijection between the finite binary trees
and the snakes.

\paragraph{Examples:} The tree $\txllf{\txllc}{ \txllf{\txllc}{\txllc} }$ is uniquely represented by the string
$babaa$ which again is uniquely represented by the snake $2,1,2,1,0$. The tree 
$\txllf{\txllf{\txllc}{ \txllf{\txllc}{\txllc} }}{\txllc}$ is uniquely represented by the string
$bbabaaa$ which again is uniquely represented by the snake $2,3,2,3,2,1,0$. \qed

\begin{definition}
For any snakes $(x_1, \ldots x_n)$ and  $(y_1, \ldots y_m)$, let
$$
(x_1, \ldots , x_n) \oplus (y_1, \ldots , y_m) \; = \; (2, x_1 + 1, x_2 + 1 , \ldots , x_n + 1, y_1, y_2, \ldots y_m)\; .
$$
\end{definition}

The proof of the next lemmas are straightforward.

\begin{lemma} \label{dddddd} 
$(x_1, \ldots , x_n) \oplus (y_1, \ldots , y_m)$ is a snake if $(x_1, \ldots x_n)$ and  $(y_1, \ldots y_m)$ are snakes.
\end{lemma}

\begin{lemma} \label{deeppurple} 
For any trees $t_1$ and $t_2$, we have
$$
\widetilde{\txllf{t_1}{t_2}} \; = \; \widetilde{t_1} \oplus \widetilde{t_2}\; . 
$$
\end{lemma}

\begin{definition}
A snake $(x_1, \ldots x_n)$  is {\em isomorphic  with a strict part}  of the snake $(y_1, \ldots y_m)$ if
there exists natural numbers $k,\ell$ such that
\begin{itemize}
\item $y_{k+1} = y_k + 1$
\item $y_{k+i} = x_i + \ell$ for $i=1,\ldots , n$
\item $y_{k+n} + 2 = y_{k+1}$.
\end{itemize}
The relation $(x_1, \ldots x_n)R(y_1, \ldots y_m)$ holds if and only if  $(x_1, \ldots x_n)= (0)$ or
$(x_1, \ldots x_n) = (y_1, \ldots y_m)$ or $(x_1, \ldots x_n)$  is  isomorphic  with a strict part $(y_1, \ldots y_m)$.
\end{definition}

\begin{lemma} \label{nazareth}   
For any trees $t_1$ and $t_2$, we have
$$
t_1 \txllr t_2 \;\; \Leftrightarrow \;\; \widetilde{t_1} R\widetilde{t_2}\; . 
$$
\end{lemma}

\subsection{Construction}

Recall that we are  working  in $  \mathsf{Seq}^+ $ and   with  $  \mathsf{N}^{  \star } $, that is, 
we consider indexed  sequences  $ f: [m, n ]  \to \mathcal{V}  $    where $ n \in \mathsf{N}^{ \star }  $. 
We use these indexed  sequences to construct an interpretation of   the tree theory $  \mathsf{T} $. 
Let   $  \mathcal{S} $ consist of all indexed  sequences $ f: [0, n ]   \to  \mathcal{V} $ 
 such that  $ n \in \mathsf{N}^{ \star } $, $ f(j) \in  \mathsf{N}^{ \star }   $  for all $  j \leq n $  and one of the following holds 
 \begin{itemize}

\item   $n= 0$ and $ f(0) = 0 $

\item  $ f(0) = 2 $, $ f(n) = 0 $,   $ f(j) \neq 0 $ for all $ j < n $ and for each  $ i < n $ we have 
$ f(i) = f(i+1) +1 $ or $ f(i+1) = f(i) +1 $.

 \end{itemize}

We define a one-to-one binary operation $ \langle \cdot , \cdot \rangle   :   \mathcal{S} \times  \mathcal{S}    \to   \mathcal{S} $. 
Assume $ f: [0, n ]   \to  \mathcal{V} $  and $ g: [0, m ]   \to  \mathcal{V} $.  
Let (recall the notation of Theorem \ref{MainTheoremIndexedSequences})   
\[
\langle f ,  g  \rangle  :=     \frac{2}{ [0, 0] }    ^{ \frown}     ( f \oplus  \frac{1}{ [0, n] }  )    ^{ \frown}   g
\       .
\]
It is easy to verify that $ \langle f ,  g  \rangle   \in \mathcal{S} $. 
That $ \langle \cdot , \cdot \rangle  $ is a function follows from Theorem \ref{MainTheoremIndexedSequences}. 
We show that  $ \langle \cdot , \cdot \rangle   $ is one-to-one. 
Let $ f^{ \prime } :  [0, n^{ \prime } ]   \to  \mathcal{V} $  and $ g^{ \prime } : [0, m^{ \prime } ]   \to  \mathcal{V} $
be elements of $  \mathcal{S} $. 
Assume $  \langle  f  ,  g    \rangle =   \langle  f^{ \prime }  ,  g^{ \prime }   \rangle  $ and $ f \neq  f^{ \prime }$. 
Then, by the editor axiom for indexed sequences  (Clause (VI) of   Theorem \ref{MainTheoremIndexedSequences}), 
one of $f$ and  $  f^{ \prime }  $ is a proper initial segment of the other. 
Without loss of generality, we may assume $f$ is a proper initial segment of    $  f^{ \prime }  $, 
that is,  that there exists an indexed sequence $ h :  [0,  \ell ]  \to   \mathcal{J} $ such that
  $ f^{ \prime }  = f  ^{  \frown }  h  $.
By definition of $ ^{ \frown} $,  we must have $ n^{ \prime } = n + \ell +1 $. 
Furthermore,    $  0 =  f( n) = f^{ \prime } ( k )  $ for some  $ k < n^{ \prime } $. . 
But this contradicts the fact that  $ f^{ \prime } ( j )  \neq 0 $ for all $ j <  n^{ \prime } $. 
Thus, $   \langle \cdot , \cdot \rangle   $ is one-to-one on  $  \mathcal{S} $.

Assume $ f: [0, n ]   \to  \mathcal{V} $  and $ g: [0, m ]   \to  \mathcal{V} $.  
Let $ f \sqsubseteq  g $ if and only if  one of the following holds: 
(i)  $n = 0 $, 
(ii)  $ f = g $, 
(iii)    there exist $ k,  \ell \in \mathsf{N}^{ \star } $ such that  
$ g(k+1) = g(k) +1 $,  $ g(k+1) = g(k+n) +2 $ and  $ g(k+i + 1 ) = f (i ) + \ell  $ for all $ 0 \leq i \leq n $. 
We show that the following holds
\[
h   \sqsubseteq  \langle  f ,  g   \rangle    \leftrightarrow     (    \   h = \langle  f ,  g   \rangle   \;   \vee  \;   h  \sqsubseteq  f   \;  \vee   \;   h  \sqsubseteq  g     \    )   
\       .
\tag{*}
\]
It is easy to see that the right-left implication holds. 
We show that the left-right implication holds. 
Assume $ h : [0, r ]   \to  \mathcal{V} $. 
By the editor axiom for indexed sequences  (Clause (VI) of   Theorem \ref{MainTheoremIndexedSequences}), 
it suffices to show that  we cannot witness $   h   \sqsubseteq  \langle  f ,  g   \rangle   $
with $ 0  \leq  k  < n+1 < n+r  \leq n+m+2 $. 
Indeed, assume this were the case. 
We have two cases:  $  \ell >0 $  and $  \ell = 0 $. 
We consider the case $  \ell  >  0 $. 
Now, there must exist   $ 0 \leq j < r $ such that   
\[
 1 = f(n) +1  =  \langle  f ,  g   \rangle  (n+1) = h (j ) + \ell 
 \     .
 \]
This  means  that $  \ell = 1 $. 
Hence, $h(j) = 0 $, which contradicts the assumption that $ h(i) \neq 0 $ for all $ i < r $ since $ h \in  \mathcal{S} $. 
Finally, we consider the case $  \ell = 0 $. 
There must exist $ 0 \leq  i \leq n  $ such that 
\[
 2 = h(0) =  \langle  f  ,  g  \rangle  ( k +1  )  = f( k+1) + 1 
 \     .
 \]
Hence, $ f (k+1 ) = 1 $. 
But, we also know, by Clause (iii)  in the  definition of $  \sqsubseteq  $,  that $  \langle  f  ,  g  \rangle  ( k  )   +1 =   \langle  f  ,  g  \rangle  ( k +1  )  $. 
Hence, $ f(k) = 0 $,  which contradicts the assumption that $ f (j ) \neq 0 $ for all $ j < n $ since $ f \in  \mathcal{S} $. 
Thus, (*) holds.

It is easy to see that $  \perp  := \frac{0}{ [0,0] }    \in  \mathcal{S} $ is not in the range of $  \langle \cdot  , \cdot  \rangle  $ and that 
$ h \sqsubseteq   \perp   $  if and only if $ h =  \perp $. 
We  have thus  constructed an interpretation of the tree theory $  \mathsf{T} $ in $  \mathsf{Seq}^+ $.

\begin{theorem}

$  \mathsf{T} $ is interpretable in $  \mathsf{Seq}^+ $, and thus also in $  \mathsf{Seq} $. 

\end{theorem}

\section{$  \mathsf{Seq} $ iterprets $  \mathsf{AST + EXT} $ }

\label{blankspisser}

Our interpretation of $  \mathsf{AST+ EXT}   $ in $  \mathsf{Seq} $  builds on   the interpretation of  $  \mathsf{T} $    in $  \mathsf{Seq} $.

\subsection{The Ideas}

We have seen that the full binary trees can be uniquely represented by the strings in $\polish$
which again can be  uniquely represented by the snakes.
For convenience we will, in what follows, identify 
the strings in $\polish$
with the snakes and regard, e.g., $babaa$ as notation for the snake $2,1,2,1,0$ (sometimes it is
 convenient to describe a snake by a giving a string from $\polish$).

We define the {\em finite trees} inductively: 
\begin{itemize}
\item $\txllc$ is a tree (the empty tree)
\item $\langle T_1, \ldots T_n \rangle$ is a tree if $T_1, \ldots T_n $ are trees (for any natural number $n\geq1$). 
\end{itemize}

Let
$F(\txllc)=\txllc$ and
$$ F(\langle T_1, \ldots  ,T_n \rangle) \; = \;  
\langle \ \ldots \langle \ \langle \ \langle \ \txllc \ , \ F(T_1) \ \rangle  \ , \  F(T_2) \ \rangle  \ , \  F(T_3) \ \rangle \ldots \ , \  F(T_n) \ \rangle \; .
$$

Then $F$ is a computable bijection between the finite trees and the binary trees (this
bijection is also used in Damnjanovic \cite{damto}).
Thus, we can represent the finite trees by snakes: every snake represents a
finite  tree, and every finite tree is  represented by one, and only 
one, snake. This works as follows:
 a snake of the form
 $$
 b^k a \beta_1\beta_2\ldots \beta_k
 $$
(where $k>0$) represents a tree of the form $\langle T_1, T_2, \ldots T_k\rangle$ where the tree $T_i$ is represented
by  $\beta_i$ (for $i=1,\ldots,k$).

Let us study a few examples. The snake
$$
\;\;\; \;\;\; \;\;\;      2,3,4,3,2,1,0 \;\;\; (bbbaaaa)
$$
represent the tree $\langle 0,0,0 \rangle$.
The snake
$$
\;\;\; \;\;\; \;\;\;  2,1,2,1,0  \;\;\; (babaa)
$$
represents the tree $\langle \langle 0 \rangle \rangle$.
The snake
$$
 \;\;\; \;\;\; \;\;\;  2,3,2, 3,4,5,4,3,2,1, 2,1,2,1,0     \;\;\;       (bbabbbaaaababaa)
$$
represents the tree
$$ \langle \   \langle 0,0,0 \rangle   \ , \ \langle \langle 0 \rangle \rangle \ \rangle \; .$$

A tree of the form $\langle T_1, T_2, T_3, T_4 \rangle$ is represented by a snake of the form
$$
2,3,4,5,4 \underbrace{ \;\;\;\;\;  \ldots \;\;\;\;\; 3}_{\mbox{\small rep. of $T_1$ }},
\underbrace{ \;\;\;\;\;  \ldots \;\;\;\;\; 2}_{\mbox{\small rep. of $T_2$ }},
\underbrace{ \;\;\;\;\;  \ldots \;\;\;\;\; ,1}_{\mbox{\small rep. of $T_3$ }},
\underbrace{ \;\;\;\;\;  \ldots \;\;\;\;\; ,0}_{\mbox{\small rep. of $T_4$ }} \; .
$$
More generally, a tree of the form $\langle T_1, \ldots , T_k \rangle$ (where $k>0$) is represented by a snake of the form
$$
2, \ldots, k+1,k, \underbrace{ \;\;\;\;\;  \ldots \;\;\;\;\; k-1}_{\mbox{\small rep. of $T_1$ }},
\underbrace{ \;\;\;\;\;  \ldots \;\;\;\;\; k-2}_{\mbox{\small rep. of $T_2$ }}, 
\;\;\;\;\;  \ldots \;\;\;\;\; 1,
\underbrace{ \;\;\;\;\;  \ldots \;\;\;\;\; 0}_{\mbox{\small rep. of $T_k$ }} 
$$
where every number preceeding $k-i$ in the sequence
$$\underbrace{ \;\;\;\;\;  \ldots \;\;\;\;\; k-i}_{\mbox{\small rep. of $T_i$ }}$$
is strictly greater than $k-i$ (if $T_i$ is the empty tree, the sequence will just consist of $k-i$).

Next we will see how we can represent hereditarily finite sets by snakes.

Let $\alpha,\beta\in \{a,b\}^*$.We write $\alpha \ll \beta$ if $\alpha$ strictly precede $\beta$ in the lexicographic ordering (any strict ordering of snakes can replace $\ll$ in the next definition).

\begin{definition}
A snake $(x_1,\ldots,x_n)$ is an {\em ordered snake} if
$n=1$ and $x_1=0$ or if $(x_1,\ldots,x_n)$ is of the form
$b^ka\alpha_1\ldots \alpha_k$ with $\alpha_i \ll \alpha_{i+1}$ (for $i=1,\ldots , k-1$).
\end{definition}

We will now define a mapping between the  hereditarily finite sets and the ordered snakes.

We map the empty set to the snake 0. 

Let  $S_1,\ldots S_k$ be hereditarily finite sets where $S_i$ and $S_j$ are
different sets when $i\neq j$. Let $\beta_1, \ldots, \beta_k\in \{a,b\}^*$ respectively, denote the ordered 
snakes mapped to
$S_1,\ldots S_k$. Let $p$ be a permutation of $\{1,\ldots, k\}$ such that 
such that $\beta_{p(1)} \ll \beta_{p(2)} \ll \ldots \ll \beta_{p(k)}$ (such a permutation exists since the sets  $S_1,\ldots S_k$ are all different). We map the set $\{ S_1,\ldots S_k \}$ to the ordered snake given by $b^ka\beta_{p(1)} \ldots \beta_{p(k)}$.

The mapping is a bijection. Thus we can represent each hereditarily  finite set with a unique ordered snake. Moreover,
every ordered snake represents a hereditarily finite set.

\subsection{Construction}

We continue to work in $  \mathsf{Seq}^+ $ and  with  $  \mathsf{N}^{ \star } $. 
We consider indexed  sequences of the form $  f :  [0, n ]  \to  \mathsf{N}^{  \star } $, 
that is,  $ n \in \mathsf{N}^{ \star }  $  and $ f(j) \in  \mathsf{N}^{ \star }  $  for all $ 0 \leq j \leq n $. 
We start by defining a linear ordering on the indexed sequences and a class $  \mathcal{W} \subseteq  \mathcal{S} $ of snakes.

Assume  $  f :  [0, n ]  \to  \mathsf{N}^{  \star } $ and  $  g :  [0, m ]  \to  \mathsf{N}^{  \star } $.
Let $f  \ll  g $ if and only if $ n < m $ or $ n = m $ and there is a least $ k \leq n $ such that 
$ f (j) = g(j) $ for all $ j < k $ and $ f(k) < g(k) $. 
Clearly  $ \ll $  is irreflexive and transitive.  We restrict the class of indexed sequences so that $ \ll $ is a linear order. 
Let  $ J_0  $ consist of all  $ n^{ \prime }  \in  \mathsf{N}^{  \star } $ such that for all  $ n \leq n^{ \prime }  $
 the class of all indexed   sequences $ f : [0, m ]  \to \mathsf{N}^{  \star } $ where $ m \leq n $   is linearly ordered by $ \ll $.
It is easy to check that $J_0 $ is closed under $0 $ and $  \mathrm{S} $ and is downward closed under $  \leq $. 
We restrict $J_0$ to a subclass $J$ that is also closed under $+$. 
Observe that $J$ is also closed under the modified subtraction function $  \dotdiv $.

Let  $  \mathcal{W} $ denote the class of all indexed  sequences  $ f : [0, n ]  \to    \mathsf{N}^{  \star }  $ such that  $ n= 0 $ and $ f(0) = 0 $  or 
 the following holds 
 \begin{itemize}

\item   $ f(0) = 2 $, $ f(n) = 0 $,   $ f(j)  > 0 $ for all $ j < n $ and for each  $ i < n $ we have 
$ f(i) = f(i+1) +1 $ or $ f(i+1) = f(i) +1 $

\item  there exists a least $ 0 \leq k^{ \star }  \leq n $  such that  
 $ f( k^{ \star } )   =  k^{ \star }    $, $ f (j+1) = f (j) +1 $ for all $ 0 \leq j +1  < k^{ \star } $ and 
 $ f ( k^{ \star }  \dotdiv  1 )  =f ( k^{ \star } ) +1 $

 \item for each $  0 \leq  \ell <   k^{ \star } $  there is a least $  0 \leq j   \leq n $ such that $ f(j) =  \ell $

 \item Assume $  0 \leq k^{ \star }   \leq m_0 < m_1  \leq m_2  < m_3  \leq n $ are such that: 
 (1) $ m_0 $ is the least index $j$ such that $ f(j) =  \ell_0 +1  $ for  $ 0 \leq \ell_0 := f(m_0)  <  k^{ \star }  $ 
 and   $ m_1 $ is the least index $j$ such that $ f(j) =  \ell_0   $; 
 (2)  $ m_2 $ is the least index $j$ such that $ f(j) =  \ell_1  +1  $ for  $ 0 \leq \ell_1 := f(m_2)  <  k^{ \star }  $ 
 and   $ m_3 $ is the least index $j$ such that $ f(j) =  \ell_1   $. 
 Then, $  \ell_0 <  \ell_1 $. 
 Furthermore, let  $ g_0 : [ 0 , m_1 - m_0 ]  \to  \mathsf{N}^{  \star } $ and 
 $ g_1 :   [ 0 , m_3 - m_2 ]  \to  \mathsf{N}^{  \star }  $  
 be such that $ f (m_0 + j + 1 ) = g_0 (j ) + \ell_0 $    and  $ f (m_2  + j + 1 ) = g_1 (j ) + \ell_1 $. 
 Then, $ g_0 \ll   g_1 $. 
 
 \end{itemize}
We refer to  the unique index  $  k^{ \star } $ as the \emph{cardinality} of (the set encoded by)  $f$. 
We define a membership relation on $  \mathcal{W} $.
Assume  $  f :  [0, n ]  \to  \mathsf{N}^{  \star } $ and  $  g :  [0, m ]  \to  \mathsf{N}^{  \star } $ are in  $  \mathcal{W} $. 
Let  $ k^{ \star }  $ be the cardinality of $f$. 
Let  $  g \in^{  \star }   f $ if and only if  there exist $ 0 \leq \ell <  k^{  \star }  $  and $  k^{  \star }   \leq  m_0 < m_1   \leq n $  such that 
\begin{itemize}

\item  $ m_0  $ is the least  $  0 \leq j   \leq n $ such that $ f(j) =  \ell  +1 $

\item  $ m_1 $ is the least  $  0 \leq j   \leq n $ such that $ f(j) =  \ell $

\item $ m_1 = m_0 + m + 1 $ and  $ f (  m_0  + j  + 1 )   = g (j )   +  \ell   $ for all $  0 \leq  j \leq   m $.

\end{itemize}
If $ \phi $ is a formula or a set-theoretic notion  in the language $  \lbrace \in  \rbrace $, we write $  \phi^{ \star } $ for its translation.

The  indexed   sequence $ \emptyset^{ \star } :  [0, 0 ]    \to    \mathsf{N}^{  \star }  $ defined my  $ \emptyset^{ \star } (0) = 0 $ is the unique  empty set 
with respect to $  \in^{  \star } $. 
We restrict $ J $ in order to get adjunction and extensionality. 
Let $ K_0  $ consist of all  
 $ n^{ \prime }  \in  J  $ such that for all  $   n   \leq n^{ \prime }  $
\begin{itemize}

\item \textup{(A)}  for any indexed  sequences $ q : [ 0, n  ]  \to \mathsf{N}^{  \star } $, 
$ f : [ 0, k  ]  \to \mathsf{N}^{  \star } $ in $  \mathcal{W} $  and 
$ g : [ 0, m  ]  \to \mathsf{N}^{  \star } $ in $  \mathcal{W} $
such that $ f \sqsubseteq  q $ and   $  g \not\in^{  \star }  f $, 
there exists an indexed sequence $ h : [ 0, k+m+2  ]  \to \mathsf{N}^{  \star } $ in in $  \mathcal{W} $    such that 
$ p \in^{  \star }  h    $ if and only if $ p \in^{  \star } f $ or $ p = g $

\item \textup{(B)}  if $ f : [ 0, n  ]  \to \mathsf{N}^{  \star } $ and  $ g : [ 0, m  ]  \to \mathsf{N}^{  \star } $  are elements of $  \mathcal{W} $ 
that  contain the same elements  with respect to $  \in^{  \star } $, 
then   $ f = g $.

\end{itemize}
It is easy to check that $ K_0 $ contains  $0$ and is downward closed under  $  \leq  $. 
We show that $K_0$ is closed under $  \mathrm{S} $. 
Assume $ n \in K_0 $. 
It suffices to check  that  $n  + 1 $ satisfies (A) and (B). 
We show that $ n  + 1 $ satisfies (A). 
So, consider $ q : [ 0, n+1  ]  \to \mathsf{N}^{  \star } $, 
$ f : [ 0, k  ]  \to \mathsf{N}^{  \star } $ in $  \mathcal{W} $  and 
$ g : [ 0, m  ]  \to \mathsf{N}^{  \star } $ in $  \mathcal{W} $
such that $ f \sqsubseteq  q $ and   $  g \not\in^{  \star }  f $. 
Since $ n \in K_0 $, $  \exists x , y \in    \mathcal{W}   \;  [   \  f  = \langle x, y  \rangle  \    ]   $   
  and  $  \sqsubseteq $ satisfies axiom $  \mathsf{T}_4 $ of  the tree theory $  \mathsf{T} $,
it suffices to consider the case $ q = f $. 
So,   $ f : [ 0, n+1  ]  \to \mathsf{N}^{  \star } $. 
Let $ 0 \leq  n_0  \leq n $ be the least index $j$  such that $ f (j ) = 1 $.
Let  $ f_0 : [ 0 , n_0 -1  ]  \to  \mathsf{N}^{  \star }  $    be such that    $ f(j+1) = f_0 (j) +1 $. 
Let  $ a : [ 0 ,  n - n_0   ]  \to  \mathsf{N}^{  \star }  $    be such that   $ f(j+1) = a  (j)  $. 
If $ a  \ll  g $, then $ h :=  \langle    f, g \rangle $ encodes $  f   \cup^{ \star }     \lbrace g \rbrace^{  \star }    $. 
Otherwise, $ g \ll  a $. 
Since $n_0 \in K_0$ and thus satisfies (A),  there exists $h_0$ that encodes  $  f_0   \cup^{ \star }     \lbrace g \rbrace^{  \star }    $. 
Then,  $ h :=  \langle    h_0 , a \rangle $  encodes $  f   \cup^{ \star }     \lbrace g \rbrace^{  \star }    $.

We show that $ n   + 1 $ satisfies (B). 
Assume $ f : [ 0, n + 1   ]  \to \mathsf{N}^{  \star } $ and  $ g : [ 0, m  ]  \to \mathsf{N}^{  \star } $  are elements of $  \mathcal{W} $ 
that  contain the same elements  with respect to $  \in^{  \star } $. 
Let $ 0 \leq  n_0  \leq n $ be the least index $j$  such that $ f (j ) = 1 $.
Let $ 0 \leq  m_0  \leq n $ be the least index $j$  such that $ g (j ) = 1 $.
Let  $ f_0 : [ 0 , n_0 -1  ]  \to  \mathsf{N}^{  \star }  $  and  $ g_0 : [ 0 , m_0 -1   ]  \to  \mathsf{N}^{  \star }  $ 
be such that 
$ f(j+1) = f_0 (j) +1 $  and $ g (j) = g_0 (j) +1 $. 
Let  $ a : [ 0 ,  n - n_0   ]  \to  \mathsf{N}^{  \star }  $  and  $ b : [ 0 , m- m_0  ]  \to  \mathsf{N}^{  \star }  $ 
be such that 
$ f(j+1) = a  (j)  $  and $ g (j) = p (j)  $. 
Since $ a, b  \in^{ \star }   f  \cap^{  \star }  g $, we must have $ a = b $  since $a$ and $ b $  are comparable with respect to $ \ll  $. 
Furthermore, $ f_0$ and $ g_0$  are elements of $  \mathcal{W} $   that  contain the same elements  with respect to $  \in^{  \star } $. 
Since  $  n_0 -1 \in K_0  $  and thus satisfies (B),  $ f_0 = g_0 $,  which implies $ f = g $
since  $ f = \langle f_0 , a \rangle $  and    $ g = \langle g_0 , b \rangle $.

We restrict $K_0$ to a subclass $K$  that is also closed under $+$. 
Let  $  \mathcal{W}^{  \star } $ consist of all $ f : [ 0, n ]   \to  \mathsf{N}^{  \star } $ in  $  \mathcal{W} $  such that $ n \in  K$.
We interpret $   \mathsf{AST+EXT} $ in $  \mathsf{Seq} $ as follows: 
(i) the domain is  $  \mathcal{W}^{  \star } $, 
(ii) the membership relation is translated as $  \in^{  \star }  $. 
It follows from the definition of $K$ that this defines an interpretation of $  \mathsf{AST+EXT} $  in  $  \mathsf{Seq}^+ $; 
indeed, observe that an element  of $ f \in  \mathcal{W}^{  \star }  $ is a subtree of $ f $ and thus also an element of 
$ \mathcal{W}^{  \star }  $ since $ \mathcal{W}^{  \star }  $ is downward closed under the restriction of $   \sqsubseteq  $  to $ \mathcal{W}  \times  \mathcal{W} $. 
Observe that the proof builds on the interpretation of  $  \mathsf{T} $ in  $  \mathsf{Seq}^+ $.

\begin{theorem}

$ \mathsf{AST+EXT}  $ is interpretable in $  \mathsf{Seq}^+ $, and thus also in $  \mathsf{Seq} $. 

\end{theorem}

\end{document}